\documentclass[graybox]{svmult}

\usepackage{mathptmx}       
\usepackage{helvet}         
\usepackage{courier}        
\usepackage{type1cm}        
\usepackage{makeidx}         
\usepackage{graphicx}        
\usepackage{multicol}        
\usepackage[bottom]{footmisc}

\usepackage{amsmath,amssymb,amscd}
\usepackage{bbm}

\makeindex             

\begin{document}

\newcommand{\bburl}[1]{\textcolor{blue}{\url{#1}}}
\newcommand{\bemail}[1]{\email{\textcolor{blue}{#1}}}

\newcommand{\hphipr}{\hphi\left(\frac{\log p}{\log R}\right)}
\newcommand{\gltp}{\lambda_{t,p}}
\newcommand{\glfp}{\lambda_f(p)}
\newcommand{\gaf}{\alpha_f}
\newcommand{\gafp}{\alpha_f(p)}
\newcommand{\gbf}{\beta_f}
\newcommand{\gbfp}{\beta_f(p)}
\newcommand{\glf}{\lambda_f}
\newcommand{\glp}{\lambda_p}
\newcommand{\gap}{\alpha_p}
\newcommand{\gbp}{\beta_p}
\newcommand{\gatp}{\alpha_{t,p}}
\newcommand{\gbtp}{\beta_{t,p}}
\newcommand{\supp}{\operatorname{supp}}
\newcommand{\notdiv}{\nmid}

\newcommand{\twocase}[5]{#1 \begin{cases} #2 & \text{{\rm #3}}\\ #4
&\text{{\rm #5}} \end{cases}   }
\newcommand{\twocaseother}[3]{#1 \begin{cases} #2 & \text{#3}\\ 0
&\text{otherwise} \end{cases}   }

\newcommand\be{\begin{equation}}
\newcommand\ee{\end{equation}}
\newcommand\bea{\begin{eqnarray}}
\newcommand\eea{\end{eqnarray}}
\newcommand{\R}{\ensuremath{\mathbb{R}}}
\newcommand{\C}{\ensuremath{\mathbb{C}}}
\newcommand{\Z}{\ensuremath{\mathbb{Z}}}
\newcommand{\Q}{\mathbb{Q}}
\newcommand{\N}{\mathbb{N}}
\newcommand{\foh}{\frac{1}{2}}  
\newtheorem{thm}{Theorem}[section]
\newtheorem{conj}[thm]{Conjecture}
\newtheorem{cor}[thm]{Corollary}
\newtheorem{lem}[thm]{Lemma}
\newtheorem{prop}[thm]{Proposition}
\newtheorem{defi}[thm]{Definition}
\newtheorem{rek}[thm]{Remark}
\newcommand{\mat}[4]{\ensuremath{\left(\begin{array}{ll}#1 &#2 \\ #3 &#4%
\end{array}\right)}}
\newcommand{\ord}{\mathop{\mathrm{ord}}}
\newcommand{\ch}{\mathop{\mathrm{char}}}
\newcommand{\Supp}{\mathop{\mathrm{supp}}}
\newcommand{\Avg}{\mathop{\mathrm{Avg}}}
\renewcommand{\ch}{\mathop{\mathrm{ch}}}
\newcommand{\SL}{\mathrm{SL}}
\newcommand{\GL}{\mathrm{GL}}
\newcommand{\iso}{\cong}
\newcommand{\eps}{\epsilon}
\newcommand{\sgn}{\mathrm{sgn}}
\newcommand{\hfrak}{\mathfrak{h}}

\newcommand{\un}{\text{U}}
\newcommand{\sy}{\text{USp}}
\newcommand{\usp}{\text{USp}}
\newcommand{\soe}{\text{SO(even)}}
\newcommand{\soo}{\text{SO(odd)}}
\newcommand{\so}{\text{O}}
\newcommand{\chiint}{\chi_{[-\foh,\foh]}}
\newcommand{\kot}[1]{ \frac{\sin \pi({#1}) }{\pi ({#1})} }
\newcommand{\kkot}[1]{ \frac{\sin \pi {#1} }{\pi {#1} } }

\renewcommand{\d}{{\mathrm{d}}} 
\newcommand{\bs}[2]{b_{i_{#1}-i_{#2}}}
\newcommand{\zeev}{Ze\'{e}v}
\newcommand{\oa}[1]{\vec{#1}}

\newcommand{\hf}{\widehat{f}}
\newcommand{\hg}{\widehat{g}}
\newcommand{\sumnz}{\sum_{n \in \Z}}
\newcommand{\summz}{\sum_{m \in \Z}}
\newcommand{\sumkz}{\sum_{k \in \Z}}
\newcommand{\sumooo}{\sum_{n=1}^\infty}
\newcommand{\intr}{\int_{\R}}
\newcommand{\intzo}{\int_0^1}

\newcommand{\ol}[1]{\overline{#1}}
\newcommand{\lp}{\left(}
\newcommand{\rp}{\right)}


\newcommand{\logpN}{ \frac{\log p}{\log(N/\pi)} }      
\newcommand{\logptN}{ 2\frac{\log p}{\log(N/\pi)} }    
\newcommand{\fologN}{O\left(\frac{1}{\log N}\right)}                       

\newcommand{\logpm}{ \frac{\log p}{\log(m/\pi)} }      
\newcommand{\logptm}{ 2\frac{\log p}{\log(m/\pi)} }    
\newcommand{\logpnp}{ \frac{\log p}{\log(\frac{N}{\pi})} }     
\newcommand{\logptnp}{ 2\frac{\log p}{\log(\frac{N}{\pi})} }    
\newcommand{\bchi}{ \overline{\chi} }                               
\newcommand{\fologm}{O\left(\frac{1}{\log m}\right)}                       
\newcommand{\fommt}{\frac{1}{m-2}}                             
\newcommand{\hphi}{\widehat{\phi}}  
\newcommand{\phiint}{\int_{-\infty}^{\infty} \phi(y) dy}
\newcommand{\FD}{\mathcal{F}} 

\newcommand{\pnddisc}{p \ {\notdiv} \ \triangle}  
\newcommand{\pddisc}{p | \triangle}          
\newcommand{\epxxi}{e^{2 \pi i x \xi}}  
\newcommand{\enxxi}{e^{-2 \pi i x \xi}} 
\newcommand{\epzxi}{e^{2 \pi i (x+iy) \xi}}
\newcommand{\enzxi}{e^{-2 \pi i (x+iy) \xi}}
\newcommand{\fologn}{O\left(\frac{1}{\log N}\right)}
\newcommand{\fllnln}{O\left(\frac{\log \log N}{\log N}\right)}
\newcommand{\flogpn}{\frac{\log p}{\log N}}
\newcommand{\plogx}{ \frac{\log p}{\log X}}
\newcommand{\plogm}{\frac{\log p}{\log M}}
\newcommand{\plogCt}{\frac{\log p}{\log C(t)}}
\newcommand{\pilogm}{\frac{\log p_i}{\log M}}
\newcommand{\pjlogm}{\frac{\log p_j}{\log M}}
\newcommand{\pologm}{\frac{\log p_1}{\log M}}
\newcommand{\ptlogm}{\frac{\log p_2}{\log M}}

\newcommand{\Norm}[1]{\frac{#1}{\sqrt{N}}}

\newcommand{\sumii}[1]{\sum_{#1 = -\infty}^\infty}
\newcommand{\sumzi}[1]{\sum_{#1 = 0}^\infty}
\newcommand{\sumoi}[1]{\sum_{#1 = 1}^\infty}

\newcommand{\eprod}[1]{\prod_p \left(#1\right)^{-1}}

\title{Some Results in the Theory of Low-lying Zeros: Determining the 1-level density, identifying the group symmetry and the arithmetic of moments of Satake parameters}

\titlerunning{Some Results in the Theory of Low-lying Zeros}


\author{Blake Mackall, Steven J. Miller, Christina Rapti, Caroline Turnage-Butterbaugh and Karl Winsor}
\institute{Corresponding author: Steven J. Miller \at Department of Mathematics \& Statistics, Williams College, Williamstown, MA 01267  \hfill \\ \bemail{sjm1@williams.edu, Steven.Miller.MC.96@aya.yale.edu}
}

%
%
\maketitle


\abstract{While Random Matrix Theory has successfully modeled the limiting behavior of many quanties of families of $L$-functions, especially the distributions of zeros and values, the theory frequently cannot see the arithmetic of the family. In some situations this requires an extended theory that inserts arithmetic factors that depend on the family, while in other cases these arithmetic factors result in contributions which vanish in the limit, and are thus not detected. In this chapter we review the general theory associated to one of the most important statistics, the $n$-level density of zeros near the central point. According to the Katz-Sarnak density conjecture, to each family of $L$-functions there is a a corresponding symmetry group (which is a subset of a classical compact group) such that the behavior of the zeros near the central point as the conductors tend to infinity agrees with the behavior of the eigenvalues near 1 as the matrix size tends to infinity. We show how these calculations are done, emphasizing the techniques, methods and obstructions to improving the results, by considering in full detail the family of Dirichlet characters with square-free conductors. We then move on and describe how we may associate a symmetry  constant to each family, and how to determine the symmetry group of a compound family in terms of the symmetries of the constituents. These calculations allow us to explain the remarkable universality of behavior, where the main terms are independent of the arithmetic, as we see that only the first two moments of the Satake parameters survive to contribute in the limit. Similar to the Central Limit Theorem, the higher moments are only felt in the rate of convergence to the universal behavior. We end by exploring the effect of lower order terms in families of elliptic curves. We present evidence supporting a conjecture that the average second moment in one-parameter families without complex multiplication has, when appropriately viewed, a negative bias, and end with a discussion of the consequences of this bias on the distribution of low-lying zeros, in particular relations between such a bias and the observed excess rank in families.}

\renewcommand{\theequation}{\thesection.\arabic{equation}}

\tableofcontents


\section{Introduction}

The purpose of this chapter is to describe some results, and the methods used to prove them, in the theory of low-lying zeros and the connections between number theory and random matrix theory. There is now an extensive literature on the subject. See for example the books \cite{Da,Ed,For,Iw,IK,KaSa2,Meh,T} and the survey articles \cite{BFMT-B,Con,KaSa1,KeSn1,KeSn2,KeSn3}, as well as \cite{Ha,FM} for popular accounts of the history of the meeting of the two fields.

Briefly, assuming the Generalized Riemann Hypothesis (GRH) the non-trivial zeros of any nice $L$-function lie on its critical line, and therefore it is possible to investigate the statistics of its normalized zeros. The work of Montgomery and Odlyzko \cite{Mon,Od1,Od2} suggested that zeros of $L$-functions in the limit are well-modeled by eigenvalues of matrix ensembles. Initially the comparison was made between number theory and the Gaussian Unitary Ensemble (GUE) with statistics such as $n$-level correlations and spacings between zeros; however, these statistics are  insensitive to finitely many zeros and in particular miss the behavior at the central point. This is a significant issue, as there are many situations in number theory where these central values are important, such as the Birch and Swinnerton-Dyer conjecture \cite{BS-D1,BS-D2}, and these statistics had the same limiting values both for different families of $L$-functions and different matrix ensembles. The reader unfamiliar with these statistics and results should see the introduction of \cite{AAILMZ,ILS} (or the introduction of any of the dissertations in low-lying zeros!) for more details.

Following the work of Katz-Sarnak \cite{KaSa1,KaSa2} a new statistic was introduced, the $n$-level density; unlike the earlier statistics this depends on the family or ensemble being studied. We mostly concentrate on the 1-level density in this paper, though see \cite{Mil1,Mil2} for some important applications of the 2-level density (which we briefly discuss later).

Let $\phi$ be an even Schwartz test function on $\R$ whose Fourier transform
\begin{equation}
  \label{eq:9}
 \hphi(y)\ =\ \int_{-\infty}^\infty \phi(x) e^{-2\pi ixy}dx
\end{equation}
has compact support. Let $\mathcal{F}_N$ be a (finite) family of $L$-functions satisfying GRH. The $1$-level density associated to
$\mathcal{F}_N$ is defined by
\begin{equation}
\label{eq:7} D_{1,\mathcal{F}_N}(\phi)\ =\ \frac1{|\mathcal{F}_N|}
\sum_{f\in\mathcal{F}_N} \sum_{j} \phi\left(\frac{\log
c_f}{2\pi}\gamma_f^{(j)}\right),
\end{equation}
where $\foh + i\gamma_f^{(j)}$ runs through the non-trivial zeros of $L(s,f)$. Here $c_f$ is the ``analytic conductor'' of $f$, and
gives the natural scale for the low zeros. As $\phi$ is Schwartz, only low-lying zeros (i.e., zeros within a distance $\ll 1/\log c_f$ of the central point $s=1/2$) contribute significantly. Thus the $1$-level density can help identify the symmetry type of the family.

Based in part on the function-field analysis where $G(\mathcal{F})$ is the monodromy group associated to the family $\mathcal{F}$, Katz and Sarnak  \cite{KaSa1,KaSa2} conjectured that for each reasonable irreducible family of $L$-functions there is an associated symmetry group $G(\mathcal{F})$ (one of the following five: unitary~$U$, symplectic~$\usp$, orthogonal~$\so$, $\soe$, $\soo$), and that the distribution of critical zeros near $1/2$ mirrors the distribution of eigenvalues near~$1$. (Similar correspondences hold for other statistics, such as the values of $L$-functions being well modeled by values of characteristic polynomials; see for example \cite{CFKRS}.) The five groups have distinguishable $1$-level densities.

To evaluate \eqref{eq:7}, one applies the explicit formula, converting sums over zeros to sums over primes. By \cite{KaSa1}, the
$1$-level densities for the classical compact groups are \begin{equation}\label{eqdensitykernels} \begin{array}{lcl}
W_{1,\soe}(x) & \ = \ & K_1(x,x) \\
W_{1,\soo}(x) & = & K_{-1}(x,x) +
\delta(x) \\
W_{1,\so}(x) & = & \foh W_{1,\soe}(x) + \foh W_{1,\soo}(x) \\
W_{1,\un}(x) & = & K_0(x,x) \\
W_{1,\sy}(x) &=& K_{-1}(x,x), \end{array}\
\end{equation}
where $K(y) = \kkot{y}$, $ K_\epsilon(x,y) = K(x-y) + \epsilon K(x+y)$ for $\epsilon = 0, \pm 1$, and $\delta(x)$ is the Dirac
delta functional. It is often more convenient to work with the Fourier transforms of the densities: \be
\begin{array}{lcl}
\widehat{W}_{1,\soe }(u) &\ =\ & \delta(u) + \foh I(u) \\
\widehat{W}_{1,\soo }(u) & = & \delta(u) - \foh I(u) + 1 \\
\widehat{W}_{1,\so }(u) & = & \delta(u) + \foh \\
\widehat{W}_{1,\un }(u) & = & \delta(u) \\
 \widehat{W}_{1,\sy}(u) & = &
\delta(u) - \foh I(u),
\end{array}\ee
where $I(u)$ is the characteristic function of $[-1,1]$. While these five densities are distinguishable for test functions
$\phi$ where the support of $\hphi$ exceeds $[-1,1]$, the three orthogonal densities are indistinguishable inside this region. While for many families of interest we cannot calculate the 1-level density beyond $[-1,1]$, we are able to uniquely associate a symmetry group by studying the 2-level densities, which are mutually distinguishable for arbitrarily small support (see \cite{Mil1,Mil2}).

Let $\mathcal{F}$ be a family of $L$-functions, and $\mathcal{F}_N$ the subset with analytic conductors $N$ (or at most $N$, or of order $N$). There is now a large body of work supporting the Katz-Sarnak conjecture that the behavior of zeros near the central point $s=1/2$ in a family of $L$-functions (as the conductors tend to infinity) agrees with the behavior of eigenvalues near $1$ of a classical compact group (unitary, symplectic, or some flavor of orthogonal). Evidence in support of this conjecture has been obtained for many  families of $L$-functions, including Dirichlet characters \cite{Gao, ER-GR, FiM, HR, LM, OS1, OS2, Rub}, elliptic curves \cite{HKS,Mil1,Mil2,Yo1}, weight $k$ level $N$ cuspidal newforms \cite{ILS,Ro,HM,MilMo,RR,Ro}, Maass forms \cite{AAILMZ,AM,GolKon}, $L$-functions attached to number fields \cite{FI,MilPe,Ya}, symmetric powers of ${\rm GL}_2$ automorphic representations \cite{Gu} and Rankin-Selberg convolutions of families \cite{DM1,DM2} to name a few.

Our purpose is to introduce the reader to some of the techniques and issues of the field. Any introduction must by necessity be brief and must sadly omit many interesting and related results. In particular, we do not discuss other models for zeros near the central point, such as the Hybrid Model (see \cite{GHK}, where $L$-functions are modeled by a partial Euler product, which encodes number theory, and a partial Hadamard product, which is believed to be modeled by matrix ensembles), or the $L$-function Ratios Conjecture \cite{CFZ1,CFZ2,CS1,CS2,FiM,GJMMNPP,HMM,Mil5,Mil7,MilMo}. We also mostly ignore the issues that arise when studying $2$-level (or higher) densities (see \cite{HM} for a determination of an alternative to the Katz-Sarnak density conjecture which facilitates comparisons between number theory and random matrix theory).

We begin in \S\ref{sec:dirichletchars} by first calculating the 1-level density of various families of Dirichlet $L$-functions. This simple family is very amenable to analysis. As such, it provides an excellent introduction to the subject and allows one to see the main ideas and techniques without becoming bogged down in technical computations. We thus show the calculations in complete detail in the hopes that doing so will help introduce newcomers to the subject.

We then turn in \S\ref{sec:convfamiliesLfns} to determining the symmetry group of convolutions of $L$-functions. Recently Shin and Templier \cite{ShTe} determined the symmetry group for many families (see also the article by Sarnak, Shin and Templier \cite{SST} in this volume); using the work of Due\~nez-Miller \cite{DM1,DM2} we are able to use inputs such as these to find the symmetry group of Rankin-Selberg convolutions, thus reducing the study of compound families to that of simple ones. In the course of our analysis we see the role lower order terms play. This leads to a nice interpretation of the remarkable universality in behavior between number theory and random matrix theory reminiscent of the universality found in the Central Limit Theorem, which we elaborate on in great detail.

We conclude in \S\ref{sec:lowerorderratesconv} with a \emph{very} brief synopsis of some work in progress on lower order terms in families of elliptic curves, and the effect they have on rates of convergence and detecting the arithmetic of the family (which is missed by the main term in the 1-level density).


\section{Families of Dirichlet $L$-Functions}\label{sec:dirichletchars}

To date, there has been significant success in showing agreement between zeros near the central point in families of $L$-functions and eigenvalues near 1 of ensembles of classical compact groups. The purpose of this section is to analyze one of the simplest examples, that of Dirichlet $L$-functions. The advantage of this calculation is that  many of the technical difficulties that plague other families are not present, and thus this provides an excellent opportunity to introduce the reader to the subject. Our first result is the following, proved by Hughes and Rudnick \cite{HR}.

\begin{thm}[$1$-Level Density for Family of Prime Conductors]\label{thm:dirichletprime}
Let $\hphi$ be an even Schwartz function with supp$(\hphi) \subset
[-2,2]$, $m$ a prime, and $\mathcal{F}_m = \{\chi: \chi$ is primitive mod
$m\}$. Then
\begin{equation}
\frac{1}{\mathcal{F}_m} \sum_{\chi \in \mathcal{F}_m} \sum_{\gamma_\chi: L(\foh +
i\gamma_\chi,\chi) = 0} \phi\left(\gamma_\chi \frac{\log(m/\pi)}{2\pi}\right)
\ = \  \int_{-\infty}^{\infty} \phi(y)dy + O\left(\frac{1}{\log
m}\right).
\end{equation} As $m\to\infty$, the above agrees only with the $N\to\infty$ limit of the $1$-level density of $N\times N$ unitary matrices.
\end{thm}

The argument below is from notes by the second named author written during the completion of his thesis \cite{Mil1}.

After proving this agreement between number theory and random matrix theory, there are two natural ways to proceed. The first is to try to extend the support. It turns out that extending the support is related to deep arithmetic questions concerning the distribution  of primes in congruence classes, which we emphasize below. While unfortunately at present there are no unconditional results, recently Fiorilli and Miller \cite{FiM} showed how to extend the support under various standard assumptions. Depending on the strength of the assumed cancelation, their results range from increasing the support up to $(-4,4)$ all the way to showing agreement for any finite support.

The other direction is to remove the restriction that the conductor is prime.

\begin{thm}[Dirichlet Characters from Square-free Numbers]\label{thm:dirichletsquarefree} Let $\mathcal{F}_{N,{\rm sq-free}}$ denote the family of primitive Dirichlet characters arising from odd square-free numbers $m \in [N,2N]$, and let $\hphi$ be an even Schwartz function with supp$(\hphi) \subset [-2,2]$ Denote the conductor of $\chi$ by $c(\chi)$. Then
\begin{eqnarray}
& &\frac{1}{|\mathcal{F}_{N,{\rm sq-free}}|} \sum_{\chi \in \mathcal{F}_{N,{\rm sq-free}}} \sum_{{\gamma_\chi}: L(\foh +
i{\gamma_\chi},\chi) = 0} \phi\left({\gamma_\chi}
\frac{\log(c(\chi)/\pi)}{2\pi}\right)\nonumber\\ & & =\ \  \int_{-\infty}^{\infty}
\phi(y)dy + O\left(\frac{1}{\log N}\right).
\end{eqnarray} As $N\to\infty$, the above agrees only with the $N\to\infty$ limit of the $1$-level density of $N\times N$ unitary matrices.
\end{thm}

While the arguments in \cite{FiM} also apply to general square-free moduli, their approach is different. We prove this result by first generalizing Theorem \ref{thm:dirichletprime} to a conductor with exactly $r$ distinct prime factors, and obtain good estimates on the error terms as a function of $r$. Theorem \ref{thm:dirichletsquarefree} then follows by controlling how many square-free numbers have $r$ factors, highlighting a common technique in the subject. We elected to show this method of proof precisely because it showcases an important technique in the subject. It is also possible to attack a fixed $m$ directly, which we do in Theorem \ref{thm:dirichletallmRef}.

\subsection{Dirichlet Characters from Prime Conductors}\label{sec:dirprimecond}

Before computing the $1$-level density of the low-lying zeros of Dirichlet $L$-functions, as one of the aims of this article is to provide a self-contained introduction to the subject we first quickly review the needed properties of Dirichlet characters and their associated $L$-functions. After these preliminaries, we use the explicit formula (see for example \cite{ILS,RS}) to relate sums of our test function over the zeros to sums of its Fourier transform weighted by Dirichlet characters. We are able to analyze these sums very easily due to the orthogonality relations of Dirichlet characters, and obtain support up to $[-2,2]$. See \cite{Da,IK} for more on Dirichlet characters.

\subsubsection{Review of Dirichlet Characters}\label{sec:reviewdirchar}

If $m$ is prime, then $(\Z / m\Z)^*$ is cyclic of order $m-1$ with generator $g$ (so any element is of the form $g^a$ for some $a$).
Let $\zeta_{m-1} = e^{2 \pi i /(m-1)}$. The principal character $\chi_0$ is
\begin{equation}
\twocase{\chi_0(k) \ = \ }{1}{if $(k,m) = 1$}{0}{if $(k,m) > 1$.}
\end{equation}

Each of the $m-2$ primitive characters are determined (because they are multiplicative) once their action on a generator $g$ is
specified. As each $\chi: (\Z / m\Z)^* \rightarrow \C^*$, for each $\chi$ there exists an $l$ such that $\chi(g) = \zeta_{m-1}^\ell$.
Hence for each $\ell$, $1 \le \ell \le m-2$, we have
\begin{equation}
\twocase{\chi_\ell(k) \ = \ }{\zeta_{m-1}^{\ell a}}{if $k \equiv g^a
\bmod m$}{0}{if $(k,m) > 0$.}
\end{equation} In most families one is not so fortunate to have such explicit formulas; these facilitate many calculations (such as proving the orthogonality relations for sums over the characters).

Let $\chi$ be a primitive character modulo $m$. Set
\begin{equation}
c(m,\chi)\ =\ \sum_{k=0}^{m-1} \chi(k) e^{2 \pi i k / m};
\end{equation}
$c(m,\chi)$ is a Gauss sum of modulus $\sqrt{m}$. The associated
$L$-function $L(s,\chi)$ (and the completed $L$-function $\Lambda(s,\chi)$) are given by
\begin{eqnarray}
L(s,\chi) & \ = \  & \prod_p (1 - \chi(p) p^{-s})^{-1} \nonumber\\
\Lambda(s,\chi) & \ = \  & \pi^{-\foh (s + \epsilon)}
\Gamma\left(\frac{s + \epsilon}{2}\right) m^{ \foh (s + \epsilon)}
L(s,\chi),
\end{eqnarray}
where
\begin{eqnarray}
\twocase{\epsilon & \ = \  &}{0}{if $\chi(-1) = \ \ 1$}{1}{if $\chi(-1)
= -1$} \nonumber\\ \Lambda(s,\chi) & = & (-i)^{\epsilon}
\frac{c(m,\chi)}{\sqrt{m}} \Lambda(1-s,\bar{\chi}).
\end{eqnarray}

Let $\phi$ be an even Schwartz function with compact support, say contained in the interval $(-\sigma,\sigma)$, and let $\chi$ be a
non-trivial primitive Dirichlet character of conductor $m$. The explicit formula\footnote{The derivation is by doing a contour integral of the logarithmic derivative of the completed $L$-function times the test function, using the Euler product and shifting contours; see \cite{RS} for details.} gives
\begin{eqnarray}
\sum_{\gamma_\chi} \phi\left({\gamma_\chi} \frac{
\log(\frac{m}{\pi})}{2\pi}\right) &\ =\ & \int_{-\infty}^{\infty}
\phi(y)dy \nonumber\\ & & -\ \sum_p \logpm \hphi\left(\logpm\right)
[\chi(p) + \bchi(p)] p^{-1/2} \nonumber\\ & & -\ \sum_p \logpm
\hphi\left(\logptm\right) [\chi^2(p) + \bchi^2(p)] p^{-1} \nonumber\\
& & +\ \fologm,
\end{eqnarray} where we are assuming GRH\footnote{It is worth noting that these formulas hold without assuming GRH. In that case, however, the zeros no longer lie on a common line and we lose the correspondence with eigenvalues of Hermitian matrices.} to write the zeros as
$\foh+i{\gamma_\chi}$, ${{\gamma_\chi}_\chi} \in \R$, and the contribution from the primes to the third and higher powers are absorbed in the big-Oh term.\footnote{A similar absorbtion holds in other families, so long as the Satake parameters satisfy $|\alpha_i(p)| \le C p^\delta$ for some $\delta < 1/6$.} Sometimes it is more convenient to normalize the zeros not by the logarithm of the analytic
conductor but rather by something that is the same to first order.\footnote{We comment on this in greater length when we consider the family of all characters with square-free modulus. Briefly, a constancy in the conductors allows us to pass certain sums through the test functions to the coefficients. This greatly simplifies the analysis of the 1-level density; unfortunately cross terms arise in the 2-level and higher cases, and the savings vanish (see \cite{Mil1, Mil2}).}
Explicitly, for $m \in [N,2N]$ we have \bea\label{eq:modifiedEF}
\sum_{\gamma_\chi} \phi\left({\gamma_\chi} \frac{
\log(\frac{N}{\pi})}{2\pi}\right) & \ =\ & \frac{\log
(m/\pi)}{\log (N/\pi)} \int_{-\infty}^{\infty} \phi(y)dy \nonumber\\
& & -\ \sum_p \logpN \hphi\left(\logpN\right) [\chi(p) + \bchi(p)]
p^{-1/2} \nonumber\\ & & -\ \sum_p \logpN \hphi\left(\logptN\right)
[\chi^2(p) + \bchi^2(p)] p^{-1} \nonumber\\ & & +\ \fologN,
\end{eqnarray} and for any subset $\mathcal{N}$ of $[N,2N]$
\be\label{eq:averagedividedlogcond} \frac1{|\mathcal{N}|}\sum_{m \in \mathcal{N}} \frac{\log
(m/\pi)}{\log (N/\pi)} \ = \ 1 + O\left(\frac1{\log N}\right). \ee

Consider $\mathcal{F}_m$, the family of primitive characters modulo a prime
$m$. There are $m-2$ elements in this family, given by
$\{\chi_\ell\}_{1 \le \ell \le m-2}$. As each $\chi_\ell$ is primitive, we
may use the Explicit Formula. To determine the $1$-level density we
must evaluate
\begin{eqnarray}\label{eq:appearancefirstsecondsums}
\int_{-\infty}^{\infty} \phi(y)dy & - & \fommt \sum_{\chi (m) \atop \chi \neq \chi_0} \sum_p \logpm \hphi\left(\logpm\right) [\chi(p) + \bchi(p)]
p^{-1/2} \nonumber\\ & - & \fommt \sum_{\chi (m) \atop \chi \neq \chi_0} \sum_p
\logpm \hphi\left(\logptm\right) [\chi^2(p) + \bchi^2(p)] p^{-1}
\nonumber\\ & + & \fologm.
\end{eqnarray}

\begin{defi}[First and Second Sums] We call the two sums in \eqref{eq:appearancefirstsecondsums} the
First Sum and the Second Sum (respectively), denoting them by
$S_1(m;\phi) $ and $S_2(m;\phi)$.
\end{defi}

The Density Conjecture states that the family average should
converge to the Unitary Density:
\begin{equation}
\lim_{m\to\infty} D_{1,\mathcal{F}_m}(\phi) \ = \ \lim_{m\to\infty} \sum_{\chi (m) \atop \chi \neq \chi_0} \sum_{\gamma_\chi} \phi\left({\gamma_\chi} \frac{
\log(\frac{m}{\pi})}{2\pi}\right) \ = \ \int_{-\infty}^\infty \phi(y)dy.
\end{equation}
We prove this for $\hphi$ supported in $[-2,2]$, which establishes Theorem \ref{thm:dirichletprime}. We break the proof into two steps. First, we show in Lemmas \ref{lem:firstsummprime} and \ref{lem:firstsummprimeBT} that the first sum does not contribute as $m\to\infty$ for such $\hphi$, and then complete the proof in Lemma \ref{lem:secondsummprime} by showing the second sum does not contribute for any finite support.

\subsubsection{The First Sum $S_1(m;\phi)$}

As one of our goals is to see how far we can get with elementary methods, in the lemma below we show that simple estimation of the prime sums allows us to determine the 1-level for support up to $(-2, 2)$, and then immediately strengthen it by using the Brun-Titchmarsh Theorem to get it for $[-2, 2]$.

\begin{lem}[Contribution from $S_1(m;\phi)$]\label{lem:firstsummprime} For ${\rm supp}(\hphi) \subset (-\sigma, \sigma)$ and $m$ prime, $S_1(m;\phi) \ll m^{\sigma/2 - 1}$, implying that this term does not contribute to the main term in the 1-level density for $\sigma < 2$. \end{lem}

\begin{proof} We must analyze
\begin{equation}
S_1(m;\phi) \ =\ \fommt \sum_{\chi (m) \atop \chi \neq \chi_0} \sum_p \logpm
\hphi\left(\logpm\right) [\chi(p) + \bchi(p)] p^{-1/2}.
\end{equation}
Since the orthogonality of the Dirichlet characters implies
\begin{equation}
\twocase{\sum_{\chi (m)} \chi(k) \ = \ }{m-1}{if $k \equiv 1 \bmod m$}{
0}{otherwise,}
\end{equation}
we have for any prime $p \neq m$ that
\begin{equation}
\twocase{\sum_{\chi (m) \atop \chi \neq \chi_0} \chi(p) \ = \ }{m-2}{if $p \equiv
1 \bmod m$}{-1}{otherwise.}
\end{equation}

Let
\begin{equation} \twocase{\delta_m(p,1) \ = \ }{1}{if $p \equiv 1 \bmod
m$}{0}{otherwise.}
\end{equation}
The contribution to the sum from $p = m$ is zero; if instead we
substitute $-1$ for $\sum_{\chi (m) \atop \chi \neq \chi_0} \chi(m)$, our error
is $O\left(1/\sqrt{m}\right)$ and hence negligible relative to the other errors.

We now calculate $S_1(m;\phi) $ with $\hphi$ an even Schwartz
function with support in $(-\sigma, \sigma)$. As the conductors are constant in the family, we may interchange the summations and first average over the family. This allows us to exploit the cancelation in sums of Dirichlet characters.
\begin{eqnarray}\label{eq:FirstSumPrimem}
S_1(m;\phi)  & \ = \  & \fommt \sum_{\chi (m) \atop \chi \neq \chi_0} \sum_p \logpm
\hphi\left(\logpm\right) [\chi(p) + \bchi(p)] p^{-1/2} \nonumber\\
& \ = \  & \fommt \sum_p \logpm \hphi\left(\logpm\right) \sum_{\chi (m) \atop \chi \neq \chi_0}[\chi(p) + \bchi(p)] p^{-1/2} \nonumber\\ & \ = \  &
\frac{2}{m-2} \sum_p \logpm \hphi\left(\logpm\right) p^{-1/2} (-1 +
(m-1)\delta_m(p,1)) \nonumber\\ & \ \ll \  &
\frac{1}{m}\sum_{p=2}^{m^\sigma} p^{-1/2} \ + \ \sum_{{p = 1 \atop p \equiv 1 (m)}}^{m^\sigma} p^{-1/2} \nonumber\\ & \ \ll \  &
\frac{1}{m}\sum_{k=1}^{m^\sigma} k^{-1/2} \ + \ \sum_{ {k = m+1 \atop k \equiv 1 (m)} }^{m^\sigma} k^{-1/2} \nonumber\\ & \ \ll \  &
\frac{1}{m}\sum_{k=1}^{m^\sigma} k^{-1/2} \ + \
\frac{1}{m}\sum_{k=1}^{m^\sigma} k^{-1/2}  \ \ll \
\frac{1}{m} m^{\sigma/2}.
\end{eqnarray}
Notice that we had to be careful with the estimates of the sum over primes congruent to 1 modulo $m$. Each residue class modulo $m$ has approximately the same sum, with the difference between two classes bounded by the first term of whichever class has the smallest element. As our numbers $k$ are of the form $\ell m+1$ for $\ell \in \{1, 2, 3, \dots\}$, the class $k \equiv 1 (m)$ has the smallest sum of the $m$ classes. Thus if we add all the classes modulo $m$ and divide by $m$, we increase the sum, justifying the above arguments.

Hence $S_1(m;\phi)  = \frac{1}{m} m^{\sigma/2}$, implying that
there is no contribution from the first sum if $\sigma < 2$.
\hfill $\Box$ \end{proof}

The next lemma illustrates a common theme in the subject: additional arithmetic information translates to increased support (and vice-versa).

\begin{lem}\label{lem:firstsummprimeBT} For ${\rm supp}(\hphi) \subset [-2, 2]$ and $m$ prime, $S_1(m;\phi) \ll 1/\log m$, implying that this term does not contribute to the main term in the 1-level density. \end{lem}

\begin{proof} Following \cite{HR} we use the Brun-Titchmarsh Theorem to improve our bound for the prime sums in \eqref{eq:FirstSumPrimem} when $\sigma = 2$. Revisiting  that calculation, we find
\begin{eqnarray}\label{eq:S1mphipreBT}
S_1(m;\phi)  & \ \ll \  &
\frac{1}{m\log m}\sum_{p=1}^{m^2} \frac{\log p}{\sqrt{p}} \ + \ \frac1{\log m}\sum_{{p = 1 \atop p \equiv 1
(m)}}^{m^2} \frac{\log p}{\sqrt{p}}.
\end{eqnarray}
The Brun-Titchmarsh theorem (see \cite{HR, MV}) states that if $x > 2m$ and $(a,m) = 1$ then \begin{eqnarray} \pi(x;m,a) \ := \ \#\{p \le x: p \equiv a (m)\} \ < \ \frac{2x}{\phi(m) \log(x/m)}. \end{eqnarray} We can trivially bound the contribution from the primes in \eqref{eq:S1mphipreBT} less than $2q$ by the arguments from Lemma \ref{lem:firstsummprime}, and for the remaining we argue as in \cite{HR}. The two sums are handled similarly. For example, for the second prime sum we have \begin{eqnarray} \frac1{\log m}\sum_{{p > 2m \atop p \equiv 1
(m)}}^{m^2} \frac{\log p}{p^{-1/2}} & \ \ll \ & \frac{1}{\log m}\int_{2m}^{m^2} \frac{\log x}{\sqrt{x}} \frac1{m} \frac{dx}{\log(x/m)}  \ \ll \ \frac1{\log m}, \end{eqnarray} proving that this term does not contribute when $\sigma = 2$. The first prime sum in \eqref{eq:S1mphipreBT} follows analogously, completing the proof.
\hfill $\Box$ \end{proof}

\subsubsection{The Second Sum $S_2(m;\phi)$}

\begin{lem}[Contribution from $S_2(m;\phi)$]\label{lem:secondsummprime} For ${\rm supp}(\hphi) \subset (-\sigma, \sigma)$ and $m$ prime, $S_2(m;\phi) \ll \sigma\frac{\log m}{m}$, implying that this term does not contribute to the main term in the 1-level density for any finite $\sigma$. \end{lem}

\begin{proof} We must analyze (for $m$ prime)
\begin{equation}
S_2(m;\phi)  \ = \  \fommt \sum_{\chi (m) \atop \chi \neq \chi_0} \sum_p \logpm
\hphi\left(\logptm\right) [\chi^2(p) + \bchi^2(p)] p^{-1}.
\end{equation}


The orthogonality relations immediately imply
\begin{equation}
S(m) \ :=\ \twocase{\sum_{\chi (m) \atop \chi \neq \chi_0} [\chi^2(p) + \bchi^2(p)]\ =\ }{2(m-2)}{if $p \equiv \pm 1 (m)$}{-2}{if $p
\not\equiv \pm 1 (m)$. }
\end{equation}
The proof is straightforward as $\chi^2(p) = \chi(p^2)$ (and similarly for $\bchi$).

%

Let
\begin{equation}
\twocase{\delta_m(p,\pm)\ =\ }{1}{if $p \equiv \pm 1$ mod
$m$}{0}{otherwise.}
\end{equation}
We argue as we did in our analysis of $S_1(m;\phi)$ in Lemma \ref{lem:firstsummprime}, and find
\begin{eqnarray}\label{eq:SecondSumPrimem}
S_2(m;\phi)  & \ = \  & \fommt \sum_{\chi (m) \atop \chi \neq \chi_0} \sum_p \logpm
\hphi\left(\logptm\right) [\chi^2(p) + \bchi^2(p)] p^{-1} \nonumber\\
& \ = \  & \fommt \sum_p \logpm \hphi\left(\logptm\right) \sum_{\chi (m) \atop \chi \neq \chi_0} [\chi^2(p) + \bchi^2(p)] p^{-1} \nonumber\\ & \ = \  & \fommt
\sum_p^{m^{\sigma/2}} \logpm \hphi\left(\logptm\right) p^{-1} [-2 +
(2m-2)\delta_m(p,\pm)] \nonumber\\ & \ \ll \  & \fommt
\sum_p^{m^{\sigma/2}} p^{-1} \ + \ \frac{2m-2}{m-2} \sum_{{p = 1 \atop p \equiv
\pm 1 (m)}}^{m^{\sigma/2}} p^{-1} \nonumber\\ & \ \ll \  & \fommt \sum_{k=1}
^{m^{\sigma/2}} k^{-1} \ + \ \sum_{ {k = m+1 \atop k \equiv 1 (m)}
} ^{m^{\sigma/2}} k^{-1} \ + \ \sum_{ {k = m-1 \atop k \equiv -1 (m)} } ^{m^{\sigma/2}} k^{-1} \nonumber\\ & \ \ll \  & \fommt
\log(m^{\sigma/2}) \ + \ \frac{1}{m} \sum_{k=1}^{m^{\sigma/2}} k^{-1} \
+ \ \frac{1}{m} \sum_{k=1}^{m^{\sigma/2}} k^{-1} \ + \
O\left(\frac{1}{m}\right) \nonumber\\ & \ \ll \  & \sigma\left(\frac{\log
m}{m} \ + \ \frac{\log m}{m} \ + \ \frac{\log m}{m}+\frac1m\right).
\end{eqnarray}
Therefore $S_2(m;\phi)  = O(\sigma\frac{\log m}{m})$, so for all fixed,
finite $\sigma$ there is no contribution.\hfill $\Box$ \end{proof}

\subsection{Dirichlet Characters from Square-free Conductors}\label{sec:dirsqfree} \setcounter{equation}{0}

We now remove the restriction that $m$ is prime and consider the more general case of square-free conductors. The purpose of this section is to highlight some of the issues that arise in the analysis of low-lying zeros in families of $L$-functions in a setting where the methods can be appreciated without being overwhelmed by technical details.

Specifically, we discuss the question of how to normalize these zeros (either locally or globally), as well as how to combine results from different cases. We find it is convenient to partition the space of characters by the number of prime factors, which we denote by $r$, of their conductors. We then generalize our bounds on the first and second sums, explicitly determining the $r$ dependence. The proof is completed by standard results on sums of the divisor function. This procedure is used in the analysis of many other families. For example, in \cite{ILS} the analysis of newforms is accomplished by using inclusion-exclusion to apply the Petersson formula to the various spaces of oldforms, removing their contributions and carefully combining the errors.

Our main result is Theorem \ref{thm:dirichletsquarefree}. As the proof is similar to the proof of Theorem \ref{thm:dirichletprime}, we content ourselves below with highlighting the differences. The first choice is how to normalize the zeros of each Dirichlet $L$-function. We can split our family by the conductor, and note that the normalization of the zeros depends \emph{only} on this quantity. Further, this number varies monotonically as we move from $N$ to $2N$. While we could normalize by the average log-conductor, or even by $\log N$, there is no additional work to rescale each $L$-function's zeros by the logarithm of the conductor. The reason is that we will break the analysis below by the size of the conductor, and our first and second sums do not contribute. The situation is different for the contribution from the Gamma factor; however, by \eqref{eq:averagedividedlogcond} there is no affect on the main terms. While the situation appears different if we looked at the $2$-level density, as then we would have cross terms and would have to deal with sums of products of logarithms of conductors and Dirichlet characters, there is no difficulty here as the conductors are constant among characters with the same moduli, and monotonically increasing with the moduli. These properties allow us to again break the analysis into characters with the same moduli. The situation is very different for one-parameter families of elliptic curves. There, we have to be significantly more careful, as these cross terms become much harder to handle. For more on these issues, see \cite{Mil1,Mil2}.


Before proving Theorem \ref{thm:dirichletsquarefree}, we first set some notation and isolate some useful results.
Fix an $r \ge 1$ and distinct, odd primes $m_1, \dots, m_r$. Let
\begin{eqnarray}
m & \ := \  & m_1 m_2 \cdots m_r \nonumber\\ M_1 & \ := \  & (m_1 - 1)(m_2 -
1) \cdots (m_r - 1) \ = \  \phi(m) \nonumber\\ M_2 & \ := \  & (m_1 - 2)(m_2
- 2) \cdots (m_r - 2).
\end{eqnarray} Note $M_2$ is the number of primitive characters mod $m$, each of
conductor $m$. For each $\ell_i \in [1,m_i - 2]$ we have the primitive
character discussed in the previous section, $\chi_{\ell_i}$. A general
primitive character mod $m$ is given by a product of these
characters:
\begin{equation}
\chi(u) \ = \  \chi_{\ell_1}(u) \chi_{\ell_2}(u) \cdots \chi_{\ell_r}(u).
\end{equation}

Let $\mathcal{F}_m = \{\chi: \chi = \chi_{\ell_1} \chi_{\ell_2} \cdots
\chi_{\ell_r}\}$. Then $|\mathcal{F}_m| = M_2$, and we are led to investigating
the following sums:
\begin{eqnarray}
S_1(m,r;\phi) & \ = \  & \frac{1}{M_2} \sum_{\chi \in \mathcal{F}_m} \sum_p \logpm
\hphi\left(\logpm\right)\frac{\chi(p) +
\bchi(p)}{\sqrt{p}} \nonumber\\ S_2(m,r;\phi) & \ = \  & \frac{1}{M_2}\sum_{\chi \in \mathcal{F}_m}  \sum_p \logpm
\hphi\left(\logptm\right)\frac{\chi^2(p) + \bchi^2(p)}{p}; \ \ \
\end{eqnarray} we have added an $r$ in the notation above to highlight the fact that $m$ has $r$ distinct odd prime factors. We first bound these two sums in terms of $r$, and then sum over $r$ to complete the proof of Theorem \ref{thm:dirichletsquarefree}.

\subsubsection{The First Sum $S_1(m,r;\phi)$ ($m$ Square-free)}

\begin{lem}[Contribution from $S_1(m,r;\phi)$]\label{lem:firstsummr} Notation as above (in particular, $m$ has $r$ factors),
\begin{eqnarray}
S_1(m,r;\phi) & \ \ll \  & \frac{1}{M_2} 2^{r} m^{\sigma/2}.
\end{eqnarray}
\end{lem}

\begin{proof}
We must study $\sum_{\chi \in \mathcal{F}_m} \chi(p)$ (the sum with $\bchi$
is handled similarly). Earlier we showed
\begin{equation}
\twocase{\sum_{\ell_i = 1}^{m_i - 2} \chi_{\ell_i}(p) \ = \ }{m_i - 1 -
1}{if $p \equiv 1 \bmod m_i$}{-1}{otherwise.}
\end{equation}

Define
\begin{equation}
\twocase{\delta_{m_i}(p,1) \ = \ }{1}{if $p \equiv 1 \bmod
m_i$}{0}{otherwise.}
\end{equation}
Then
\begin{eqnarray}
\sum_{\chi \in \mathcal{F}_m} \chi(p) & \ = \  &  \sum_{\ell_1 = 1}^{m_1 - 2} \cdots
\sum_{\ell_r = 1}^{m_r - 2} \chi_{\ell_1}(p) \cdots \chi_{\ell_r}(p)
\nonumber\\ & \ = \  & \prod_{i=1}^{r} \sum_{\ell_i = 1}^{m_i - 2}
\chi_{\ell_i}(p)  \ = \  \prod_{i=1}^{r}(-1 + (m_i -
1)\delta_{m_i}(p,1)).
\end{eqnarray}
Let us denote by $k(s)$ an s-tuple $(k_1,k_2,\dots,k_s)$ with $k_1 <
k_2 < \cdots < k_s$. This is just a subset of $\{1,2,\dots,r\}$.
There are $2^r$ possible choices for $k(s)$. We use these to
expand the above product. Define
\begin{equation}
\delta_{k(s)}(p,1) \ = \  \prod_{i=1}^s \delta_{m_{k_i}}(p,1).
\end{equation}
If $s = 0$ we set $\delta_{k(0)}(p,1) = 1$ for all $p$. Then
\begin{equation}
\prod_{i=1}^{r}(-1 + (m_i - 1)\delta_{m_i}(p,1)) \ = \  \sum_{s=0}^{r}
\sum_{k(s)} (-1)^{r-s} \delta_{k(s)}(p,1) \prod_{i=1}^s (m_{k_i} -
1).
\end{equation}
Let $h(p) = 2 \logpm \hphi\left(\logpm\right)  \ll  ||\hphi||$. Then
\begin{eqnarray}
S_1 & \ = \  & \sum_p^{m^\sigma} \foh h(p) p^{-1/2} \frac{1}{M_2}
\sum_{\chi \in \mathcal{F}} [\chi(p) + \bchi(p)] \nonumber\\ & \ = \  &
\sum_p^{m^\sigma} h(p) p^{-1/2} \frac{1}{M_2} \sum_{s=0}^{r}
\sum_{k(s)} (-1)^{r-s} \delta_{k(s)}(p,1) \prod_{i=1}^s (m_{k_i} -
1) \nonumber\\ & \ \ll \  & \sum_p^{m^\sigma} p^{-1/2} \frac{1}{M_2}
\left(1 + \sum_{s=1}^{r} \sum_{k(s)} \delta_{k(s)}(p,1)
\prod_{i=1}^s (m_{k_i} - 1)\right).
\end{eqnarray}
Observing that $m/M_2 \le 3^r$ we see the $s=0$ sum contributes
\begin{equation}
S_{1,0} \ = \  \frac{1}{M_2} \sum_p^{m^\sigma} p^{-1/2} \ \ll \  3^r
m^{\sigma/2 - 1},
\end{equation}
which is negligible for $\sigma < 2$, though it is also bounded by $m^{\sigma/2-1}/M_2$. Now we study
\begin{eqnarray}
S_{1,k(s)} \ = \  \frac{1}{M_2}  \prod_{i=1}^s (m_{k_i} - 1)
\sum_p^{m^\sigma} p^{-1/2}\delta_{k(s)}(p,1).
\end{eqnarray}
The effect of the factor $\delta_{k(s)}(p,1)$ is to restrict the
summation to primes $p \equiv 1 (m_{k_i})$ for $k_i \in k(s)$. The
sum will increase if instead of summing over primes satisfying the
congruences we sum over all numbers $n$ satisfying the congruences
(with $n \ge 1 + \prod_{i=1}^s m_{k_i}$). As the sum is now over integers and not primes, we can use basic uniformity properties
of integers to bound it. We are summing integers mod $\prod_{i=1}^s
m_{k_i}$, so summing over integers satisfying these congruences is
basically just $\prod_{i=1}^s (m_{k_i})^{-1}$ $\sum_{n=1}^{m^\sigma}
n^{-1/2}$ $=$ $\prod_{i=1}^s (m_{k_i})^{-1} m^{\sigma/2}$. We
can do this as the sum of the reciprocals from the residue classes
of $\prod_{i=1}^s m_{k_i}$ differ by at most their first term.
Throwing out the first term of the class $1 + \prod_{i=1}^s m_{k_i}$
makes it have the smallest sum of the $\prod_{i=1}^s m_{k_i}$
classes, so adding all the classes and dividing by $\prod_{i=1}^s
m_{k_i}$ increases the sum. Hence (recalling $m/M_2 \le 3^r$)
\begin{eqnarray}
S_{1,k(s)} & \ \ll \  & \frac{1}{M_2}  \prod_{i=1}^s (m_{k_i} - 1)
\prod_{i=1}^s (m_{k_i})^{-1} m^{\sigma/2}  \ \ll \
3^r m^{\sigma/2 - 1},
\end{eqnarray} though  it is also bounded by $m^{\sigma/2-1}/M_2$. Therefore, for all $s$ the $S_{1,k(s)}$ contribute $3^r m^{\sigma/2 - 1}$. There are $2^r$ choices, yielding
\begin{equation}
S_1 \ \ll \  6^r m^{\sigma/2 - 1},
\end{equation}
which is negligible as $m$ goes to infinity \emph{for fixed r} if $\sigma <
2$. If instead we do not use  $m/M_2 \le 3^r$ we obtain a bound of $O(2^r m^{\sigma/2} / M_2)$. \hfill $\Box$ \end{proof}

The worst errors occur when $m$ is the product of the first $r$ primes. Let $p_i$ denote the $i$\textsuperscript{th} prime. The Prime Number Theorem implies for $r$ large that
\begin{eqnarray}
\log m \ = \  \sum_{p \le p_r} \log p \ \sim \  p_r.  \end{eqnarray} As $p_r \sim r \log r$, we find $\log m \sim r \log r$ or $r \sim \log m/\log\log m$. Thus \begin{equation} 6^r\ \sim\ e^{r \log 6}\ \sim \ m^{\log 6 / \log\log m}. \end{equation} While this is $o(m^\epsilon)$ for any $\epsilon > 0$, this estimate is wasteful when $m$ has few prime factors. For example, if $m = 10^{50}$ then $m^{\log 6/\log\log m} \sim m^{0.3775}$, which is sizable. We thus prefer to leave the estimate of $S_1(m,r;\phi)$ as a function of $r$, and then average over the number of square-free integers with exactly $r$ distinct odd prime factors. Such a division will lead to significantly better results for the family of square-free conductors.

\subsubsection{The Second Sum $S_2(m,r;\phi)$ ($m$ Square-free)}

\begin{lem}[Contribution from $S_2(m,r;\phi)$]\label{lem:secondsummr} Notation as above (in particular, $m$ has $r$ factors),
\begin{eqnarray}
S_2(m,r;\phi) & \ \ll \  & \frac{1}{M_2} 3^{r} m^{1/2}.
\end{eqnarray}
\end{lem}

\begin{proof}
We must study $\sum_{\chi \in \mathcal{F}} \chi^2(p)$ (the sum with
$\bchi$ is handled similarly). Earlier we showed
\begin{equation}
\twocase{\sum_{\ell_i = 1}^{m_i - 2} \chi_{\ell_i}^2(p) \ = \ }{m_i - 1
- 1}{if $p \equiv \pm 1 \bmod m_i$}{-1}{otherwise.}
\end{equation}
Then
\begin{eqnarray}
\sum_{\chi \in \mathcal{F}} \chi^2(p) & \ = \  &  \sum_{\ell_1 = 1}^{m_1 - 2}
\cdots \sum_{\ell_r = 1}^{m_r - 2} \chi_{\ell_1}^2(p) \cdots
\chi_{\ell_r}^2(p) \nonumber\\ & \ = \  & \prod_{i=1}^{r} \sum_{\ell_i =
1}^{m_i - 2} \chi_{\ell_i}^2(p) \nonumber\\ & \ = \  & \prod_{i=1}^{r}(-1
+ (m_i - 1)\delta_{m_i}(p,1) + (m_i - 1)\delta_{m_i}(p,-1)).
\end{eqnarray}

Instead of having $2^r$ terms as in the first sum, now we have $3^r$. Let $k(s)$ be as before, and
let $j(s)$ be an s-tuple of $\pm 1$'s. As $s$ ranges from $0$ to $r$ we get each of the $3^r$ possibilities, as for a fixed $s$ there are ${r \choose s}$ choices for $k(s)$, each of these having $2^s$ choices for $j(s)$ (note $\sum_{s=0}^r 2^s {r \choose k} = (1+2)^r$). Let $h(p) = \logptm \hphi \left(\logptm\right) \ll ||\hphi||$.
Define
\begin{equation}
\delta_{k(s)}(p,j(s)) \ = \  \prod_{i=1}^s \delta_{m_{k_i}}(p,j_i).
\end{equation}
Then
\begin{equation}
\sum_{\chi \in \mathcal{F}} \chi^2(p) \ = \  \sum_{s=0}^{r} \sum_{k(s)}
\sum_{j(s)} (-1)^{r-s} \delta_{k(s)}(p,j(s)) \prod_{i=1}^s
(m_{k_i} - 1).
\end{equation}
Therefore
\begin{eqnarray}
S_2 & \ = \  & \frac{1}{M_2} \sum_p \logpm \hphi\left(\logptm\right)
p^{-1} \sum_{\chi \in \mathcal{F}}[\chi^2(p) + \bchi^2(p)] \nonumber\\ & \ = \  &
\frac{1}{M_2} \sum_p h(p) \sum_{s=0}^r \sum_{k(s)} \sum_{j(s)}
p^{-1} (-1)^{r-s} \delta_{k(s)}(p,j(s)) \prod_{i=1}^s (m_{k_i} - 1)
\nonumber\\ & \ \ll \  & \frac{1}{M_2} \sum_p \sum_{s=0}^r \sum_{k(s)}
\sum_{j(s)} p^{-1} \delta_{k(s)}(p,j(s)) \prod_{i=1}^s (m_{k_i} - 1)
\nonumber\\ & \ = \  & \sum_{s=0}^r \sum_{k(s)} \sum_{j(s)} S_{2, k(s),
j(s)}.
\end{eqnarray}

The term where $s=0$ is handled easily (recall $m/M_2 \le 3^r$):
\begin{equation}
S_{2,0,0} \ = \  \frac{1}{M_2} \sum_p^{m^\sigma} p^{-1} \ \ll \  3^r
\frac{\log m^\sigma}{m}
\end{equation} (we could also bound it by $\sigma \log(m) / M_2$).

We would like to handle the terms for $s \neq 0$ analogously as
before. The congruences on $p$ from $k(s)$ and $j(s)$ force us to
sum only over certain primes mod $\prod_{i=1}^s m_{k_i}$, with
each prime satisfying $p \ge m_{k_i} \pm 1$. We increase the sum
by summing over all integers satisfying these congruences. As each
congruence class mod $\prod_{i=1}^s m_{k_i}$ has basically the
same sum, we can bound our sum over primes satisfying the
congruences $k(s), j(s)$ by $\prod_{i=1}^s (m_{k_i})^{-1}
\sum_{n=1}^{m^\sigma} n^{-1} = \prod_{i=1}^s (m_{k_i})^{-1} \log
m^{\sigma}$.

There is one slight problem with this argument. Before each prime
was congruent to $1$ mod each prime $m_{k_i}$, hence the first
prime occurred no earlier than at $1 + \prod_{k=1}^{s} m_{k_i}$.
Now, however, some primes are congruent to $+1$ mod $m_{k_i}$ while others are congruent to to $-1$, and it is possible the first such prime occurs
before $\prod_{k=1}^{s} m_{k_i}$.

For example, say the prime is congruent to $+1$ mod $11$, and $-1$
mod $3,5,17$. We want the prime to be greater than $3 \cdot 5
\cdot 11 \cdot 17$, but $3 \cdot 5 \cdot 17 - 1$ is congruent to
$-1$ mod $3,5,17$ and $+1$ mod $11$. (Fortunately it equals 254,
which is composite.)

So, for each pair $(k(s),j(s))$ we handle all but the possibly
first prime as we did in the First Sum case. We now need an
estimate on the possible error for low primes. Fortunately, there
is at most one for each pair, and as our sum has a $1/p$,
we can expect cancelation if it is large.

Fix now a pair (remember there are at most $3^r$ pairs). As we
never specified the order of the primes $m_i$, without loss of
generality (basically, for notational convenience) we may assume
that our prime $p$ is congruent to $+1$ mod $m_{k_1} \cdots
m_{k_a}$, and $-1$ mod $m_{k_{a+1}} \cdots m_{k_s}$.

The contribution to the second sum from the possible low prime in
this pair is
\begin{eqnarray}
\frac{1}{M_2} \frac{1}{p} \prod_{i=1}^s (m_{k_i} - 1).
\end{eqnarray}
How small can $p$ be? The $+1$ congruences imply that $p \equiv 1
(m_{k_1} \cdots m_{k_a}$), so $p$ is at least $m_{k_1} \cdots
m_{k_a} + 1$. Similarly the $-1$ congruences imply $p$ is at least
$m_{k_{a+1}} \cdots m_{k_s} - 1$. Since the product of these two
lower bounds is greater than $\prod_{i=1}^s (m_{k_i} - 1)$, at least
one must be greater than $\left(\prod_{i=1}^s (m_{k_i} - 1)
\right)^{1/2}$. Therefore the contribution to the second sum from the
possible low prime in this pair is bounded by (remember $m/M_2 \le
3^r$)
\begin{eqnarray}
\frac{1}{M_2} \left( \prod_{i=1}^s (m_{k_i} - 1) \right)^{1/2} \ \le \
\frac{m^{1/2}}{M_2} \ \le \  3^r m^{-1/2}.
\end{eqnarray}
Combining this with the estimate for the primes larger than
$\prod_{i=1}^s (m_{k_i} - 1)$ yields
\begin{equation}
S_{2, k(s), j(s)} \ \ll \ 3^r m^{-1/2} + \frac{3^r}{m} \log
m^{\sigma},
\end{equation}
yielding (as there are $3^r$ pairs)
\begin{equation}
S_2 \ = \  \sum_{s=0}^r \sum_{k(s)} \sum_{j(s)} S_{2, k(s), j(s)} \ \ll \
9^r m^{-1/2};
\end{equation} if  we don't use $m/M_2 \le 3^r$ we find a bound of $3^r m^{1/2}/M_2$.
\hfill $\Box$ \end{proof}

\subsubsection{Proof of Theorem \ref{thm:dirichletsquarefree}}\label{sec:dirsqfree2}
\setcounter{equation}{0}

We now extend the results of the previous sections to consider
the family $\mathcal{F}_{N;{\rm sq-free}}$ of all primitive characters whose conductor is
an odd square-free integer in $[N,2N]$. Some of the bounds below
can be improved, but as the improvements do not increase the range
of convergence, they will only be sketched.

\begin{proof}[Proof of Theorem \ref{thm:dirichletsquarefree}]
First we calculate the number of primitive characters arising from
odd square-free numbers $m \in [N,2N]$. Let $m = m_1 m_2 \cdots
m_r$. Then $m$ contributes $(m_1 - 2) \cdots (m_r - 2)$
characters. On average we might expect the number of characters to be of order $N$, and as a positive percent of numbers are
square-free, we expect there to be on the order of $N^2$ characters.

Instead we prove there are at least $N^2 / \log^2 N$ primitive characters in the family; as we are winning by power savings and not logarithms, the $\log^2 N$ factor is harmless. There are at least $N / \log^2 N \ + \ 1$ primes in the interval. For each prime $p$ (except possibly the first) we have $p - 2 \ge N$. Hence there are at least $N \cdot
\frac{N}{\log^2 N} = N^2 / \log^{2}N$ primitive characters. Let $M
= |\mathcal{F}_{N;{\rm sq-free}}|$. Then
\begin{equation}
M \ \ge \ N^2 \log^{-2}N \ \ \ \ \Rightarrow \ \ \ \  \frac{1}{M}
\ \le \ \frac{\log^2 N}{N^2}.
\end{equation}

We recall the results from the previous section. Fix an odd
square-free number $m \in [N,2N]$, and say $m$ has $r = r(m)$
factors. Before we divided the First and Second sums by $M_2 =
(m_1 - 2) \cdots (m_r - 2)$, as this was the number of primitive
characters in our family. Now we divide by $M$. Hence the
contribution to the First and Second Sums from this $m$ is
\begin{eqnarray}
S_1(m,r;\phi) & \ \ll \  & \frac{1}{M} 2^{r(m)} m^{\sigma/2} \nonumber\\
S_2(m,r;\phi) & \ \ll \  & \frac{1}{M} 3^{r(m)} m^{1/2}.
\end{eqnarray}
Note that $2^{r(m)} = \tau(m)$, the number of divisors of $m$. While
it is possible to prove
\begin{eqnarray}
\sum_{n \le x} \tau^\ell(n) \ \ll \ x(\log x)^{2^\ell - 1}
\end{eqnarray}
the crude bound
\begin{eqnarray}
\tau(n) \ \le \ c(\epsilon) n^\epsilon
\end{eqnarray}
yields the same region of convergence. Note $3^{r(m)} \le
\tau^2(m)$. Therefore by  Lemma \ref{lem:firstsummr} the contributions to the first sum is
majorized by
\begin{eqnarray}
\sum_{ {m = N \atop m \ {{\rm square-free}}} }^{2N}
S_1(m,r;\phi) & \ \ll \  & \sum_{m = N}^{2N} \frac{1}{M} 2^{r(m)}
m^{\sigma/2} \nonumber\\ & \ \ll \  & \frac{1}{M} N^{\sigma/2}
\sum_{m = N}^{2N} \tau(m) \nonumber\\ & \ \ll \  & \frac{1}{M} N^{\sigma/2} c(\epsilon) N^{1+\epsilon} \nonumber\\ & \ \ll \  & \frac{\log^2
N}{N^2} N^{\sigma/2} c(\epsilon) N^{1+\epsilon} \nonumber\\ & \ \ll \
& c(\epsilon) N^{\foh \sigma + \epsilon - 1} \log^2 N.
\end{eqnarray}
For $\sigma < 2$, choosing $\epsilon < 1 - \foh \sigma$ yields $S_1$
goes to zero as $N$ tends to infinity. For the second sum, Lemma \ref{lem:secondsummr} bounds it by
\begin{eqnarray}
\sum_{ {m = N \atop m \ {\rm square-free}} }^{2N} S_2(m,r;\phi)
& \ \ll \  & \sum_{m = N}^{2N} \frac{1}{M} 3^{r(m)} m^{1/2}
\nonumber\\ & \ \ll \  & \frac{1}{M} N^{1/2} \sum_{m = N}^{2N} \tau^2(m)
\nonumber\\ & \ \ll \  & c(\epsilon) \frac{\log^2 N}{N^2} N^{1/2} N^{1+2
\epsilon} \nonumber\\ & \ \ll \  & c(\epsilon) N^{2 \epsilon - \foh}
\log^2 N,
\end{eqnarray}
which converges to zero as $N$ tends to infinity for all $\sigma$ and completes the proof. \hfill $\Box$ \end{proof}

\subsection{Dirichlet Characters from a Fixed Modulus}

We thank the referee for the following theorem and proof, which extends Theorem \ref{thm:dirichletprime} to the family of Dirichlet characters for any fixed modulus.

\begin{thm}[Dirichlet Characters from a Fixed Modulus]\label{thm:dirichletallmRef} Let $\mathcal{F}_m$ denote the family of primitive Dirichlet characters arising from a fixed $m$, and let $\hphi$ be an even Schwartz function with supp$(\hphi) \subset (-2,2)$ Denote the conductor of $\chi$ by $c(\chi)$. Then
\begin{equation}
\frac{1}{\phi(m)} \sum_{\chi (m) \atop \chi \neq \chi_0} \sum_{{\gamma_\chi}: L(\foh +
i{\gamma_\chi},\chi) = 0} \phi\left({\gamma_\chi}
\frac{\log(c(\chi)/\pi)}{2\pi}\right) \ = \  \int_{-\infty}^{\infty}
\phi(y)dy + O\left(\frac{1}{\log m}\right).
\end{equation} As $m\to\infty$, the above agrees only with the $m\to\infty$ limit of the $1$-level density of $m \times m$ unitary matrices.
\end{thm}

\begin{proof} We argue similarly as in the proof of Theorem \ref{thm:dirichletprime}. From Equation (3.8) of \cite{IK} we have \begin{equation} \sum_{\chi (m)} \chi(p) \ = \ \sum_{d|(p-1,m)} \phi(d) \mu(m/d). \end{equation} We can now bound the first prime sum, $S_1(m;\phi)$:  \begin{eqnarray} S_1(m;\phi) & \ = \ &  \frac1{\phi(m)} \sum_{d|m} \phi(d)\mu(m/d) \sum_{p \equiv 1 (d)} \frac{\log p}{\log(m/\pi)} \hphi\left(\frac{\log p}{\log(m/\pi)}\right) \frac{2{\rm Re}(\chi(p))}{\sqrt{p}}\nonumber\\ & \ \ll \ & \frac{m^{\sigma/2}}{\phi(m)} \sum_{d|m}\frac{\phi(d)}{d} \ \le \ \frac{\tau(m)}{\phi(m)} m^{\sigma/2}, \end{eqnarray} which is $O(1/\log m)$, completing the proof. \hfill $\Box$
\end{proof}

\begin{rek} We could argue as in the proof of Theorem \ref{thm:dirichletallmRef}, and by applying t{rek}he Brun-Titchmarsh Theorem extend the support to $[-2, 2]$. \end{rek}


\section{Convolutions of Families of $L$-Functions}\label{sec:convfamiliesLfns}

The analysis of Dirichlet $L$-functions in \S\ref{sec:dirichletchars} highlights the general framework for determining the behavior of the low-lying zeros in a family and identifying the corresponding symmetry group. In this section we describe how to find the symmetry group of a compound family in terms of its constituent pieces. In order to view these results in the proper context, we first briefly summarize the procedure used in most works to calculate 1-level densities, and refer the reader to \cite{SaShTe} in this volume for a more detailed treatment. 

These calculations break down into three steps. The first step is to understand and control conductors. In most families analyzed to date they are either constant, or monotonically increasing. Their importance stems from the fact that their logarithm controls the spacing of zeros near the central point, and constancy or monotonicity allows us to pass sums over the family past the test function to the Fourier coefficients. When these properties fail, the computations are significantly harder. A notable exception is in one-parameter families of elliptic curves over $\mathbb{Q}(T)$, where for $t \in [N,2N]$ variations in the logarithms of the conductors, from $\log(N^d)$ to $\log(cN^d)$, greatly complicates the analysis and requires careful sieving.

The second step is the classic explicit formula, which relates sums of our test function $\phi$ at the zeros of the $L$-functions to sums of its Fourier transform $\hphi$ at the primes (weighted by the coefficients of the $L$-function). This is very similar to the role the Eigenvalue Trace Lemma plays in random matrix theory. While we wish to understand the eigenvalues of a matrix, it is the matrix elements where we have information; the Eigenvalue Trace Lemma allows us to pass from knowledge of the matrix coefficients (which we have) to knowledge of the eigenvalues (which we desire). The explicit formulas in number theory play a similar role.

The explicit formula is useless, however, unless we have a way to execute the resulting sums. The final step is to use an averaging formula for weighted sums of $L$-function coefficients. Examples here include the orthogonality relations of Dirichlet characters, the Petersson formula for holomorphic cusp forms, and the Kuznetsov trace formula for Maass forms. Unfortunately, as our family becomes more complicated the averaging formulas become harder to use, and often yield smaller support. This can be seen in comparison of some recent work (such as \cite{GolKon,MaTe,ShTe}).

The goal for the remainder of this section is to discuss how to identify the corresponding symmetry group for a family of $L$-functions, and to discuss the role the Fourier coefficients play in the rate of convergence of the 1-level density to the scaling limits of ensembles from the classical compact groups.

\subsection{Identifying the Symmetry Group of a Family}

Determining the corresponding symmetry group for a family of $L$-functions is one of the hardest questions in the subject. In many cases we cannot compute the 1-level density for large enough support to distinguish between the three orthogonal candidates (though we can uniquely determine which works by looking at the 2-level density). In many situations we are able to argue by analogy with a function field analogue, where the situation is clearer and the answer arises from the monodromy group. Another approach is to work with the Sato-Tate measure of the family as carried out in
\cite{SaShTe}.

A folklore conjecture stated that the symmetry was determined by the sign of the functional equations. For example, if all the signs were odd then the family had to have SO(odd) symmetry. If the signs are all even then there are two candidates: Symplectic and SO(even). Initially many thought that SO(even) symmetry happened when there was a corresponding family with odds signs that was being ignored (for example, splitting the family of weight $k$ and level $N>1$ cuspidal newforms by sign and ignoring the forms with odd sign), and that if there were no corresponding family with odd signs then the symmetry would be Symplectic. Due\~nez and Miller \cite{DM1} disproved this conjecture by analyzing a family suggested by Sarnak: $\{L(s,\phi \times {\rm sym}^2 f): f \in H_k\}$, where $\phi$ is a fixed even Hecke-Maass cusp form and $H_k$ is a Hecke eigenbasis for the space of holomorphic cusp forms of weight $k$ for the full modular group. Their proof involved finding the symmetry group of a Rankin-Selberg convolution in terms of the symmetry groups of the constituents. They generalized their argument to many families in \cite{DM2}. We quickly sketch the main ideas of that argument, and then conclude this section with an interpretation of convergence to the limiting densities in the spirit of the Central Limit Theorem.

We first need some standard notation and results.

\begin{itemize}

\item $\pi$: A cuspidal automorphic representation on ${\rm GL}_n$.\\ \

\item $Q_\pi>0$: The analytic conductor of $L(s,\pi) = \sum \lambda_\pi(n) / n^s$.\\ \

\item By GRH\footnote{The definition of the 1-level density as a sum of a test function at scaled zeros is well-defined even if GRH fails; however, in that case the zeros are no longer on a line and we thus lose the ability to talk about spacings between zeros. Thus in many of the arguments in the subject GRH is only used to interpret the quantities studied, though there are exceptions (in \cite{ILS} the authors use GRH for Dirichlet $L$-functions to expand Kloosterman sums).} the non-trivial zeros are $\frac12+i\gamma_{\pi,j}$. \\ \

\item $\{\alpha_{\pi,i}(p)\}_{i=1}^n$: The Satake parameters, and $\lambda_\pi(p^\nu)
= \sum_{i=1}^n \alpha_{\pi,i}(p)^\nu$. Thus the $p^\nu$-th coefficient of $L(s,\pi)$ is the $\nu$-th moment of the Satake parameters.\\ \

\item $L(s,\pi) = \sum_n \frac{\lambda_\pi(n)}{n^s} = \prod_p \prod_{i=1}^n \left(1 -
\alpha_{\pi,i}(p) p^{-s}\right)^{-1}$.

\end{itemize}

The explicit formula, applied to a given $L(s,\pi)$, yields \begin{equation}
  \sum_j g\left(\gamma_{\pi,j}\frac{\log Q_\pi}{2\pi}\right)\ =\
  \widehat{g}(0) -
  2\sum_{p} \sum_{\nu = 1}^\infty\widehat{g}\left(\frac{\nu\log p}{\log Q_\pi}\right)
  \frac{\lambda_\pi(p^\nu)\log p}{p^{\nu/2}\log Q_\pi}.
\end{equation}

For ease of exposition, we assume the conductors in our family are constant,\footnote{It is easy to handle the case where the conductors are monotone by rescaling the zeros by the average log-conductor; as remarked many times above the general case is more involved.} and thus $Q_\pi = Q$ say. Thus in calculating the 1-level density we can push the sum over our family $\mathcal{F}_N$ through the test function; here $\mathcal{F}_N$ are all forms in our infinite family $\mathcal{F}$ with some restriction involving $N$ on the conductor (frequent choices are the conductor equals $N$, lives in an interval $[N, 2N]$, or is at most $N$). The 1-level density is then found by taking the limit as $N\to\infty$. We rescale the zeros by $\log R$, where $R$ is closely related to $Q$ (it sometimes differs by a fixed, multiplicative constant; this extra flexibility simplifies some of the resulting expressions for various families).

We also assume sufficient decay in the $\lambda_\pi(p^\nu)$'s so that the sum over primes with $n \ge 3$ converges; this is known for many families. Determining the 1-level density, up to lower order terms which we will return to later, is equivalent to analyzing the $N\to\infty$ limits of \bea S_1(\mathcal{F}_N) & \ := \  & -2 \sum_{p}
\widehat g\left(\ \frac{\log p}{\log R}\right) \frac{\log p}{\sqrt{p}
\log R}\ \left[\frac{1}{|\mathcal{F}_N|} \sum_{\pi \in \mathcal{F}_N}
\lambda_\pi(p)\ \right]
\nonumber\\ S_2(\mathcal{F}_N) & \ := \ & -2 \sum_{p}\widehat g\left(2\frac{\log p}{\log
R}\right)\ \frac{\log p}{p \log R}\ \ \left[
\frac{1}{|\mathcal{F}_N|} \sum_{\pi \in \mathcal{F}_N} \lambda_\pi(p^2)\right].
\eea As
\be \lambda_\pi(p^\nu) \ = \ \alpha_{\pi,1}(p)^\nu + \cdots
+ \alpha_{\pi,n}(p)^\nu, \ee we see that only the first two moments of the Satake parameters enter the calculation. The sum over the remaining powers, \be S_\nu(\mathcal{F}_N)  \ := \  -2  \sum_{\nu = 3}^\infty \sum_p \widehat g\left(\nu\frac{\log p}{\log
R}\right)\ \frac{\log p}{p^{\nu/2} \log R}\ \ \left[
\frac{1}{|\mathcal{F}_N|} \sum_{\pi \in \mathcal{F}_N} \lambda_\pi(p^\nu)\right], \ee is $O(1/\log R)$ under the Ramanujan Conjectures.\footnote{The Satake parameters $|\alpha_{\pi,i}|$ are bounded by $p^\delta$ for some $\delta$, and it is conjectured that we may take $\delta = 0$. While this conjecture is open in general, for many forms there is significant progress towards these bounds with some $\delta < 1/2$. See for example recent work of Kim and Sarnak \cite{K,KSa}. For  our purposes, we only need to be able to take $\delta < 1/6$, as such an estimate and trivial bounding suffices to show that the sum over all primes and all $\nu \ge 3$ is $O(1/\log R)$.}

To date, the only families where the first sum $S_1(\mathcal{F}_N)$ is not negligible are elliptic curve families with rank. The presence of non-zero terms here require trivial modifications to the classical random matrix ensembles, and effectively in the limit only result in additional independent zeros at the central point. Thus, if the family has rank $r$, the scaling limit is that of a block diagonal matrix, with an $r\times r$ identity matrix in the upper left, and then an $(N-r) \times (N-r)$ matrix in the lower right (with the other two rectangular blocks zero).

We introduce a symmetry constant for the family, $c_\mathcal{F}$, as follows: \be c_\mathcal{F} \ := \
\lim_{N\to\infty}\frac1{|\mathcal{F}_N|} \sum_{\pi\in\mathcal{F}_N} \lambda_\pi(p^2), \ee which is the limit  of the average second  moment of the Satake parameters. The corresponding classical compact group is Unitary if $c_\mathcal{F}$ is 0, Symplectic if $c_\mathcal{F} = 1$, and Orthogonal if $c_\mathcal{F} = -1$. Equivalently, $c_{\mathcal{F}} = 0$ (respectively, 1 or -1) if the family $\mathcal{F}$ has Unitary (respectively, Symplectic or Orthogonal) symmetry.

\subsection{Identifying the Symmetry Group from Rankin-Selberg Convolutions}

In this section we assume we have two families of $L$-functions where we can determine the corresponding symmetry group. Under standard assumptions (which are proven in many cases), the Rankin-Selberg convolution exists and it makes sense to talk about the symmetry group of the family. We assume for simplicity below that $\pi_2$ is not the representation contragredient to $\pi_1$, and thus the $L$-function below will not have a pole, though with more book-keeping this case can readily be handled. The Satake parameters of the convolution $\pi_{1,p}\times\pi_{2,p}$ are
\begin{equation}\label{eq:satakeparamsprod}
  \{\alpha_{\pi_1\times \pi_2,k}(p)\}_{k=1}^{nm}\ =\ \{\alpha_{\pi_1,i}(p)
  \cdot\alpha_{\pi_2,j}(p)\}_{{1 \le i \le n \atop 1 \le j \le m}}.
\end{equation}

The main result is that the symmetry of the new compound family is beautifully and simply related to the symmetry of the constituent pieces. See \cite{DM2} for a statement of which families are nice (examples include Dirichlet $L$-functions and ${\rm GL}_2$ families).

\begin{thm}[Due\~nez-Miller \cite{DM2}] If $\mathcal{F}$ and $\mathcal{G}$ are nice families of $L$-functions, then $c_{\mathcal{F} \times \mathcal{G}} = c_{\mathcal{F}} \cdot c_{\mathcal{G}}$.\end{thm}

\begin{proof}[Sketch  of the proof]
From \eqref{eq:satakeparamsprod}, we find that the moments of the Satake parameters for $\pi_{1,p}\times\pi_{2,p}$ are
\begin{equation}
 \sum_{k=1}^{nm} \alpha_{\pi_1\times \pi_2,k}(p)^\nu \ =\ \sum_{i=1}^n \alpha_{\pi_1,i}(p)^\nu
  \sum_{j=1}^m \alpha_{\pi_2,j}(p)^\nu.
\end{equation}
Thus, if $\pi_1 \in \mathcal{G}_N$ and $\pi_2 \in \mathcal{G}_M$, we find \bea c_{\mathcal{F}\times\mathcal{G}} & \ = \ & \lim_{N,M \to \infty} \frac{1}{|\mathcal{F}_N|\ |\mathcal{G}_M|} \sum_{\pi_1 \in \mathcal{F}_N \atop \pi_2 \in \mathcal{G}_M} \lambda_{\pi_1 \times \pi_2}(p^2)\nonumber\\ & \ = \ & \lim_{N,M \to \infty} \frac{1}{|\mathcal{F}_N|} \sum_{\pi_1 \in \mathcal{F}_N} \lambda_{\pi_1}(p^2) \frac1{|\mathcal{G}_M|} \sum_{\pi_2 \in \mathcal{G}_N} \lambda_{\pi_2}(p^2) \ = \ c_{\mathcal{F}} c_{\mathcal{G}}. \eea The first sum is handled similarly, and the higher moments do not contribute by assumption on the family (the definition of a good family includes sufficient bounds towards the Ramanujan conjecture to handle the $\nu \ge 3$  terms). \end{proof}

\subsection{Connections to the Central  Limit Theorem}\label{sec:connCLT}

We end this section by interpreting our results in the spirit of the Central Limit Theorem, which we hope will shed some light on the universality of results.

Interestingly, random matrix theory does not seem to know about arithmetic. By this we mean that very different families of $L$-functions converge to one of five flavors (unitary, symplectic, or one of the three orthogonals), independent of the arithmetic structure of the family. It doesn't matter if we have quadratic Dirichlet characters or the symmetric square of ${\rm GL}_2$-forms; we see symplectic behavior. Similarly it doesn't matter if our family of elliptic curves have complex multiplication or not, or instead are holomorphic cusp forms of weight $k$ or Maass forms; we see orthogonal behavior.\footnote{There are some situations where arithmetic enters. The standard example is that in estimating moments of $L$-functions one has a product $a_k g_k$, where $a_k$ is an arithmetic factor coming from the arithmetic of the form and $g_k$ arises from random matrix theory. See for example \cite{CFKRS,KeSn1,KeSn2}.}

One of the first places this universality was noticed was in the work of Rudnick and Sarnak \cite{RS}, who showed for suitable test functions that the $n$-level correlations of zeros arising from a fixed cuspidal automorphic representation agreed with the Gaussian Unitary Ensemble. The cause of their universality was that the answer was governed by the first and second moments of the Fourier coefficients, and explained why the behavior of zeros far from the central point was the same for all $L$-functions.

We have a similar explanation for the behavior of the zeros near the central point. Our universality is due to the fact that the main term of the limiting behavior depends only on the first two moments of the Satake parameters, which to date have very few possibilities. The effect of the higher moments are felt only in the $\nu \ge 3$ terms, which (under the Generalized Ramanujan Conjectures)  contribute $O(1/\log R)$. While these contributions vanish in the limit, they can be felt in \emph{how} the limiting density  is approached.

Notice how similar this is to the Central Limit Theorem, which in its simplest form states that the normalized sum of independent random variables drawn from the same nice distribution (finite moments suffice) converges to being normally distributed. If the mean $\mu$ and the variance $\sigma^2$ of a random variable $X$ are finite, we can always study instead the standardized random variable $Z = (X-\mu)/\sigma$, which has mean 0 and variance 1. Thus the first `free' moment of our density is the third (or fourth if the distribution is symmetric). A standard proof is to look at the Fourier transform of the $N$-fold convolution, Taylor expand, and then show that the inverse Fourier transform converges to the Gaussian. The higher moments emerge only in the error terms, and while they have no contribution as $N\to\infty$ they do affect the rate in which the density of the convolution approaches the Gaussian.

Thus, for families of $L$-functions the higher moments of the Satake parameters help control the convergence to random matrix theory, and can depend on the arithmetic of family. This leads to the exciting possibility of isolating lower order terms in 1-level densities, and seeing the arithmetic of the family emerge.

Unfortunately, it is often very hard to isolate these lower order terms from other errors. For example, Due\~nez and Miller \cite{DM2} convolve two families of elliptic curves with ranks $r_1$ and $r_2$, and see a potential lower order term of size $r_1 r_2$ divided by the logarithm of the conductor. Thus, while this looks like a lower order term which is highly dependent on the arithmetic of the family, there are other error terms which can only be bounded by larger quantities (though we believe these bounds are far from optimal and that this product term should be larger in the limit). We discuss some of these issues in more detail in the concluding section.


\section{Lower Order Terms and Rates of Convergence}\label{sec:lowerorderratesconv}

In this section we discuss some work (see \cite{Mil3, Mil6}) on lower order terms in families of elliptic curves, though similar results can be done for other families (especially families of Dirichlet $L$-functions \cite{FiM} or cusp forms \cite{MilMo}). We first report on some families where these lower order terms have been successfully isolated (which is different than the example from convolving two families with rank from \S\ref{sec:connCLT}), and end with some current research about finer properties of the distribution of the Satake parameters in families of elliptic curves and lower order terms.

\subsection{Arithmetic-Dependent Lower Order Terms in Elliptic Curve Families}\label{sec:arithmdeplowerordterms}

The results below are from \cite{Mil6}, where many families of elliptic curves were studied. For families of elliptic curves, it is significantly easier to calculate and work with $\lambda_E(p)$ (which is an integer and computable via sums of Legendre symbols) then the Satake parameters $\alpha_{E,1}(p)$ and $\alpha_{E,2}(p)$. We thus first re-express the formula for the 1-level density to involve sums over the $\lambda_E$'s, and then give several families with lower order terms depending on the arithmetic.


It is often convenient to study weighted moments (for example, in \cite{ILS} much work is required to remove the harmonic weights, which facilitated applications of the Petersson formula). For a family $\mathcal{F}$ and a weight function $w$ define  \bea A_{r,\mathcal{F}}(p) & \ := \ & \frac1{W_R(\mathcal{F})}\sum_{f \in \mathcal{F} \atop f \in S(p)} w_R(f)\lambda_f(p)^r \nonumber\\ A_{r,\mathcal{F}}'(p) & \ :=\ & \frac1{W_R(\mathcal{F})}\sum_{f \in \mathcal{F} \atop f \not\in S(p)} w_R(f)\lambda_f(p)^r  \nonumber\\ S(p) & \ := \ & \{f \in \mathcal{F}: p\ \notdiv C_f\}, \eea where $C_f$ is the conductor of $f$  (when doing the computations, there are sometimes differences at primes dividing the conductor, and it is worth isolating their contribution). The main difficulty in determining the 1-level density is evaluating \be S(\mathcal{F}) \ := \ -\ 2 \sum_p \sum_{m=1}^\infty \frac1{W_R(\mathcal{F})} \sum_{f\in\mathcal{F}}w_R(f)\frac{\alpha_{f,1}(p)^m + \alpha_{f,2}(p)^m}{p^{m/2}} \frac{\log p}{\log R}\ \hphi\left(m\frac{\log p}{\log R}\right),\ee where we are assuming we have ${\rm  GL}_2$ forms.

The following alternative expansion for the explicit formula from \cite{Mil6} is especially tractable for families of elliptic curves:
\bea   S(\mathcal{F}) & \ = \ & -\ 2 \sum_p \sum_{m=1}^\infty \frac{A_{m,\mathcal{F}}'(p)}{p^{m/2}} \frac{\log p}{\log R}\
\hphi\left(m\frac{\log p}{\log R}\right)\nonumber\\ & &
-2\hphi(0)\sum_p \frac{2A_{0,\mathcal{F}}(p)\log p} {p(p+1)\log R}\ +\
2\sum_p \frac{2A_{0,\mathcal{F}}(p)\log p}{p\log R}\ \hphi\left(
2\frac{\log p}{\log R}\right)\nonumber\\
& &-2\sum_p \frac{A_{1,\mathcal{F}}(p)}{p^{1/2}} \frac{\log p}{\log R}\
\hphipr + 2 \hphi(0) \frac{A_{1,\mathcal{F}}(p)
(3p+1)}{p^{1/2}(p+1)^2}\frac{\log p}{\log R} \nonumber\\ & &
-2\sum_p \frac{A_{2,\mathcal{F}}(p)\log p}{p\log R}\ \hphi\left(2\frac{\log
p}{\log R}\right) + 2\hphi(0)\sum_p \frac{A_{2,\mathcal{F}}(p)(4p^2+3p+1)\log
p}{p(p+1)^3\log R} \nonumber\\ & & - 2\hphi(0)\sum_{p}
\sum_{r=3}^\infty \frac{A_{r,\mathcal{F}}(p)p^{r/2}(p-1)\log
p}{(p+1)^{r+1}\log R} \ + \ O\left(\frac1{\log^3 R}\right)
\nonumber\\ & \ =\  &  S_{A'}(\mathcal{F}) + S_0(\mathcal{F})+S_1(\mathcal{F})+S_2(\mathcal{F}) +
S_A(\mathcal{F})+O\left(\frac1{\log^3 R}\right).\eea

Letting $\widetilde A_\mathcal{F}(p) \ := \ \frac1{W_R(\mathcal{F})} \sum_{f \in S(p)} w_R(f)
\frac{\glf(p)^3}{p+1 - \lambda_f(p)\sqrt{p}}$, by the geometric
series formula we may replace $S_A(\mathcal{F})$ with $S_{\tilde A}(\mathcal{F})$,
where \be S_{\tilde A}(\mathcal{F}) \ = \ - 2\hphi(0)\sum_{p}
\frac{\widetilde A_\mathcal{F}(p) p^{3/2}(p-1)\log p}{(p+1)^3 \log R}. \ee


We now state some results (see \cite{Mil6} for the proofs). For comparison purposes we start with the family of cuspidal newforms, as this family is significantly easier to calculate and serves as a good baseline. In reading the formulas below, it is important to note that the contributions from the smaller primes are significantly more than those from the larger primes. For elliptic curves the primes 2 and 3 often behave differently;  while they will have no affect on the main term, they will strongly influence the lower order terms.

In the subsections below, we assume the logarithms of the conductors are of size $\log R$, so that we are comparing zeros of similar size. In all families of elliptic curves we start with an elliptic curve over $\Q(T)$, and then form a one-parameter family by looking at the specializations from setting $T$ equal to integers $t$. \\ \

\subsubsection{$\mathcal{F}_{k,N}$ the family of even weight $k$ and prime level $N$ cuspidal newforms, or just the forms with even (or odd) functional equation.}


Up to $O(\log^{-3} R)$, as $N\to\infty$ for test functions $\phi$ with $\supp(\hphi) \subset (-4/3, 4/3)$  the
(non-conductor) lower order term for either of these families is \be C \cdot 2\hphi(0)/\log R,  \ee with $C \approx -1.33258$. In other words, the difference between the Katz-Sarnak prediction and the 1-level density has a lower order term of order $1/\log R$, with the next correction $O(1/\log^3 )$. Note the lower order corrections are independent of the distribution of the signs of the functional
equations, and the weight $k$.

\subsubsection{CM example, with or without forced torsion: Specializations of $y^2 = x^3 + B(6T+1)^\kappa$ over $\Q(T)$, with $B \in \{1,2,3,6\}$ and $\kappa \in \{1,2\}$.}


This family of elliptic curves has complex multiplication. We consider the sub-family obtained by sieving and restricting $T$ so that $(6T+1)$ is $(6/\kappa)$-power free. If $\kappa = 1$ then all values of $B$ give the same result, while if $\kappa=2$ then the four values of $B$ have different lower order corrections. Note if $\kappa =2$ and $B=1$ then there is a forced torsion point of order three, $(0,6T+1)$.

Up to errors of size $O(\log^{-3} R)$, the (non-conductor) lower order terms are again of size $C\cdot 2\hphi(0)/\log R$; we give numerical approximations for the $C$'s for various choices of $B$ and $\kappa$: \bea B
= 1,\kappa = 1: &\ \ & -2.124 \cdot 2\hphi(0)/\log R, \nonumber\\ B
= 1,\kappa = 2: &\ \ & -2.201 \cdot 2\hphi(0)/\log R, \nonumber\\
B = 2,\kappa = 2: &\ \ & -2.347 \cdot 2\hphi(0)/\log R \nonumber\\
B = 3,\kappa = 2: &\ \ & -1.921 \cdot 2\hphi(0)/\log R \nonumber\\
B = 6,\kappa = 2: &\ \ & -2.042 \cdot 2\hphi(0)/\log R.  \eea


\subsubsection{CM example, with or without rank: Specializations of $y^2 = x^3 - B(36T+6)(36T+5)x$ over $\Q(T)$, with $B\in \{1,2\}$.}


We consider another complex multiplication family. If $B=1$ the family has rank 1 over $\Q(T)$, while if $B=2$ the family has rank 0. We consider the sub-family obtained by sieving to $(36T+6)(36T+5)$ is cube-free. Again we find a lower order term of size $C \cdot 2\hphi(0)/\log R$, with next term of size $O(1/\log^3  R)$. The most important difference between these two families is the contribution from the $S_{\widetilde{\mathcal{A}}}(\mathcal{F})$ terms, where the $B=1$ family is approximately $-.11 \cdot 2\hphi(0)/\log R$, while the $B=2$ family is approximately $.63\cdot 2\hphi(0)/\log R$. This large difference is due to biases of size $-r$ in the Fourier coefficients $a_t(p)$ in a one-parameter family of rank $r$ over $\Q(T)$.

The main term of the average moments of the $p$\textsuperscript{th} Fourier coefficients are given by the complex multiplication analogue of Sato-Tate in the limit, for each $p$ there are lower order correction terms which depend on the rank.


\subsubsection{Non-CM Example: Specializations of $y^2 = x^3 - 3x+12T$ over $\Q(T)$.}


Up to $O(\log^{-3} R)$, the (non-conductor) lower order correction for this family is $C \cdot 2\hphi(0)/\log R$, where $C \approx -2.703$. Note this answer is very different than the family of weight $2$ cuspidal newforms of prime level $N$.

\subsection{Second Moment Bias in One-Parameter Families of Elliptic Curves}\label{sec:secondmomentbiasEC}

In \S\ref{sec:arithmdeplowerordterms} we saw lower order terms to the 1-level density for families of elliptic curves which depended on the arithmetic of the family. In this section we report on work on progress on possible family-dependent lower order terms to the second moment of the Fourier coefficients in families of elliptic curve $L$-functions; see \cite{MMRW} for a more complete investigation of these families, and Appendix A for some initial results on other families. We then conclude in \S\ref{sec:lowerorderandonelevel} by exploring the implications such a bias would have on low-lying zeros (in particular, in understanding the excess rank phenomenon).

We have observed an interesting property in the average second moments of the Fourier coefficients of elliptic curve $L$-functions over $\Q(T)$. Specifically, consider an elliptic curve $\mathcal{E}: y^2 = x^3 + A(T) x + B(T)$ over $\Q(T)$, where $A(T), B(T)$ are polynomials in $\Z[T]$ and the curve $E_t$ (obtained by specializing $T$ to $t$) has coefficient $a_t(p)$ (of size $2\sqrt{p}$) in the series expansion of its $L$-function. Define the average second moment $A_2(p)$ for the family by
\begin{equation}
	A_{2}(p)\ := \ \frac1{p} \sum_{t \bmod p} a_t(p)^2
\end{equation} (where for notational convenience we are suppressing the subscript $\mathcal{E}$ on $A_2$, as the family is fixed). Michel \cite{Michel95} proved that \begin{equation}
	A_{2}(p)\ = \ p^2 + O(p^{3/2})
\end{equation}
for families of elliptic curves with non-constant $j$-invariant $j(T)$, and cohomological arguments show that the lower-order terms\footnote{These bounds cannot be improved, as Miller \cite{Mil3} found a family where there is a term of size $p^{3/2}$.} are of sizes $p^{3/2}$, $p$, $p^{1/2}$, and $1$. In every case that we have proven or numerically analyzed, the following conjecture holds.

\begin{conj}[Bias Conjecture] For any family of elliptic curves $\mathcal{E}$ over $\mathbb{Q}(T)$, the largest lower order term in the second moment of $\mathcal{E}$ which does not average to $0$ is on average negative. Explicitly, from Michel \cite{Michel95} we have \be A_2(p) \ = \ p^2 + \beta_{3/2}(p) p^{3/2} + \beta_1(p) p + \beta_{1/2} p^{1/2} + \beta_0(p) \ee where each $\beta_r(p)$ is of order 1; when we write the second moment thusly the first $\beta_r(p)$ term which does not average to zero will average to a negative value.
\end{conj}

Below, we give several proven cases of the Bias Conjecture and some preliminary numerical evidence supporting the conjecture. We have made several additional observations about the terms in the second moments, though we do not know if these always hold.

\begin{itemize}
	\item In families with constant $j$-invariant, the largest term is on average $p^2$ (rather than exactly $p^2$), and the Bias Conjecture appears to hold similarly.
    \item Every explicit second moment expression has a non-zero $p^{3/2}$ term or a non-zero $p$ term (or both). The term of size $p^{3/2}$ always averages to $0$, and the term of size $p$ is always on average negative.
    \item In many cases the terms of size $p^{3/2}$ and/or $p$ are governed by the values of an elliptic curve coefficient, that is, a sum of the form
    \begin{equation}
    	\sum_{x \bmod p} \left(\frac{ax^3 + bx^2 +cx + d}{p} \right),
    \end{equation}
possibly squared, cubed, or multiplied by $p$, et cetera.
\end{itemize}

Rosen and Silverman \cite{RosenSilverman} proved that the negative bias in the first moments is related to the rank of family by
\begin{equation}
	\lim_{X \rightarrow \infty} \frac{1}{X} \sum_{p \leq X} A_{1}(p) \frac{\log{p}}{p}\ =\ \mathrm{rank}\mathcal{E}(\mathbb{Q}(T)).
\end{equation}
It is natural to ask whether the bias in the second moments is also related to the family rank. We are currently investigating this. More generally, we can ask if higher moments are also biased and if this bias is also related to the rank of the family.

\subsubsection{Evidence: Explicit Formulas}\label{sec:evidence}

We have proven the conjecture for a variety of specific families and some restricted cases, and list a few of these below; these are a representative subset of families we have successfully studied, and we are currently investigating many more. The average bias refers to the average value of the coefficient of the largest lower order term not averaging to 0 (which in all of our cases is the $p$ term).

\begin{lem}\label{lem:tconst}
	Consider elliptic curve families of the form $y^2 = ax^3 + bx^2 + cx + d + eT$. These families have rank $0$ over $\mathbb{Q}(T)$, and for primes $p > 3$ with $p \nmid a,e$ and $p \nmid b^2 - 3ac$,
   	\begin{equation}
    	A_{2}(p)\ =\ p^2 - p\left(1 + \left(\frac{b^2 - 3ac}{p}\right) + \left(\frac{-3}{p}\right) \right).
    \end{equation}
    These families obey the Bias Conjecture with an average bias of $-1$ in the $p$ term.
\end{lem}

\begin{lem}\label{lem:tlin}
	Consider families of the form $y^2 = ax^3 + bx^2 + (cT + d)x$. These families have rank $0$, and for primes $p > 3$ with $p \nmid a,b,c$,
    \begin{equation}
    	A_{2}(p)\ =\ p^2 - p\left(1 + \left(\frac{-1}{p}\right) \right).
    \end{equation}
    These families obey the Bias Conjecture with an average bias of $-1$ in the $p$ term.
\end{lem}

\begin{lem}
\label{lem:tnx}
	Consider families of the form $y^2 = x^3 + T^n x$. These families have rank $0$, and for primes $p > 3$,
    \begin{equation}
    	A_{2}(p)\ =\ \begin{cases} \left(p - 1\right) \left(\sum_{x(p)} \left(\frac{x^3 + x}{p} \right) \right)^2 & \mathrm{if\ } n \equiv 0 (2) \\ \left(p^2 - p\right)\left(1 + \left(\frac{-1}{p}\right)\right) & \mathrm{if\ } n \equiv 1 (2). \end{cases}
    \end{equation}
    These families obey the Bias Conjecture with an average bias of $-4/3$ for $n \equiv 0 (2)$ and $-1$ for $n \equiv 1 (2)$ in the $p$ term.
\end{lem}

\begin{lem}
\label{lem:tn}
	Consider families of the form $y^2 = x^3 + T^n$. These families have rank $0$, and for primes $p > 3$,
    \begin{equation}
    	A_{2}(p)\ =\ \begin{cases} \left(p - 1\right) \left(\sum_{x(p)} \left(\frac{x^3 + 1}{p} \right) \right)^2 & \mathrm{if\ } n \equiv 0 (3) \\ p^2 - p \left(1 + \left(\frac{-3}{p}\right)\right) & \mathrm{if\ } n \equiv 1 (3) \\ p^2 - p & \mathrm{if\ } n \equiv 2 (3). \end{cases}
    \end{equation}
    These families obey the Bias Conjecture with an average bias of $-4/3$ for $n \equiv 0 (3)$ and $-1$ for $n \equiv 1, 2 (3)$ in the $p$ term.
\end{lem}

\begin{lem}
\label{lem:wash}
	Consider families of the form $y^2 = x^3 + Tx^2 + (mt - 3m^2)x - m^3$ for $m$ a non-zero integer. These families have rank $0$ for $m$ non-square and rank $1$ for $m$ a square, and for primes $p > 3$,
    \begin{equation}
    	A_{2}(p)\ =\ p^2 - p\left(2 + 2\left(\frac{-3}{p}\right)\right) - 1.
    \end{equation}
	These families obey the Bias Conjecture with an average bias of $-2$.
\end{lem}

Lemmas \ref{lem:tconst} and \ref{lem:tlin} prove the Bias Conjecture for a large number of families studied by Fermigier in \cite{Fe2}. A more systematic study of Fermigier's families (which is in progress \cite{MMRW}) will help determine whether the bias in second moments is correlated to the family rank. Lemmas \ref{lem:tnx} and \ref{lem:tn} provide examples of complex-multiplication families where the Bias Conjecture holds. Lemma \ref{lem:wash} proves the conjecture for a family with an unusual distribution of signs, providing stronger evidence for the conjecture.

\subsubsection{Numerical Data}\label{sub:data}
The following lemma is useful for analyzing Fermigier's rank 1 families \cite{Fe2}.

\begin{lem}
Consider families of the form $y^2 = ax^3 + cx^2 + (dT + e)x + g$. For $p \nmid d,g$,
\begin{equation}
    A_{2}(p)\ =\ p^2 + pc_{1}(p) - pc_{0}(p),
\end{equation}
where $c_{0}(p)$ is the number of roots of the congruence $2ax^3 + cx^2 - g \equiv 0 (p)$ and $c_{1}(p) = \sum_{x,y:axy^2 + (ax^2 + cx)y - g \equiv 0 (p)} \left(\frac{xy}{p}\right)$.
\end{lem}

We are not able to explicitly determine the $c_{1}(p)$ term in general, but the data in Table \ref{tab:r1-3/2} suggests that on average this term is $0$. We averaged these coefficients over the 6000th to the 7000th primes, and all averages are very small in absolute value. Thus, we believe that these families obey the Bias Conjecture with an average bias of $c_{0}(p)$, which in most cases is about $1$. We collected additional data on rank 2 families, and found similar evidence from these families that the $p^{3/2}$ term coefficient is on average $0$.

\begin{table}
\centering
	\begin{tabular}{| c | c | c |} \hline
    	Family	& Average$(c_{1}(p))$	& Average$(c_{0}(p))$	\\ \hline
        $y^2 = 4x^3 - 7x^2 + 4tx + 4$		& 0.0068	& 0.974	\\ \hline
        $y^2 = 4x^3 + 5x^2 + (4t-2)x + 1$	& -0.0176	& 1.005	\\ \hline
        $y^2 = 4x^3 + 5x^2 + (4t+2)x + 1$	& -0.0174	& 1.005	\\ \hline
        $y^2 = 4x^3 + x^2 + (4t+2)x + 1$	& 0.0399	& 0.993	\\ \hline
        $y^2 = 4x^3 + x^2 + 4tx + 4$		& 0.0068	& 0.985	\\ \hline
        $y^2 = 4x^3 + x^2 + (4t+6)x + 9$	& -0.0113	& 1.988	\\ \hline
        $y^2 = 4x^3 + 4x^2 + 4tx + 1$		& 0.0072	& 0.974	\\ \hline
        $y^2 = 4x^3 + 5x^2 + (4t+4)x + 4$	& 0.0035	& 1.012	\\ \hline
        $y^2 = 4x^3 + 4x^2 + 4tx + 9$		& 0.0256	& 1.005	\\ \hline
        $y^2 = 4x^3 + 5x^2 + 4tx + 4$		& 0.0043	& 1.005	\\ \hline
        $y^2 = 4x^3 + 5x^2 + (4t+6)x + 9$	& -0.0143	& 1.037	\\ \hline
    \end{tabular}
\caption{Averages of $p^{3/2}$ term coefficients in rank 1 families}
\label{tab:r1-3/2}
\end{table}

We also collected numerical data for several families that were too complicated to analyze explicitly. We used two averaging statistics,
\begin{equation}
	\mathbb{E}_{p}\left(\frac{A_{2}(p)-p^2}{p^{3/2}}\right), \quad \quad \mathbb{E}_{p}\left(\frac{A_{2}(p)-p^2}{p}\right),
\end{equation}
where the averages are taken over some range of primes. These statistics are meant to quantify the average bias in the cases where the largest lower term is of size $p^{3/2}$ and $p$, respectively. For these families, we calculated the second moment for the 100th to 150th primes. In every case, the running $p^{3/2}$-normalized average was small in magnitude, further suggesting that the $p^{3/2}$ term coefficient is on average $0$. In most families, the $p$-normalized statistic revealed a clear negative average bias, but two families showed a positive $p$-normalized average bias. The problem behind these statistics is the rate of decay of the $p^{3/2}$ term. In order for these statistics to reliably detect an average bias, the average coefficient of the $p^{3/2}$ term would need to exhibit enough cancelation that in the limit it would be smaller than the conjectured bias coming from the lower order terms. This is only a heuristic, but it suggests that we need to improve this method of analyzing general families. The positive average families were positive overall but had a negative average on the second half of the primes. However, here we feel as though we are trying to force out a negative average. For several families that support the conjecture, we tried averaging only over the second half of our sample to see if the bias was still negative in this reduced sample, and it was in each case.

In the last section we discuss connections of the negative bias with excess rank. It is important to note, however, that it is the smallest primes that contribute the most. Thus while there may be a negative bias overall, at the end of the day what might matter most is what occurs for the primes 2 and 3 (and other small primes).

\subsection{Biases and Excess Rank}\label{sec:lowerorderandonelevel}

We end by very briefly discussing an application of the conjectured negative bias in the second moments to the observed excess rank in families. For more details, see \cite{Mil3}. The purpose of this section is to show how the arithmetic in lower order terms can be used as a possible explanation for some interesting phenomena. The 1-level density, with an appropriate test function, is used to obtain upper bounds for the average rank; there were several papers using essentially the 1-level density for this purpose before Katz and Sarnak isolated the 1-level density as a statistic to study independent of rank estimation. We show that lower order terms arising from arithmetic contribute for finite conductors and require a very slight change in the upper bound of the average rank. Of course, this is not a proof of a connection between these factors and the average rank, as all we can show is that these affect the upper bound; however, it is worth noting the role they play in such calculations. For more on finite models and the behavior of elliptic curve zeros, see \cite{DHKMS1, DHKMS2}.

For a one-parameter family of elliptic curves $\mathcal{E}$ of rank $r$ over $\Q(T)$, assuming the Birch and
Swinnerton-Dyer conjecture by Silverman's specialization theorem
eventually all curves $E_t$ have rank at least $r$, and under
natural standard conjectures a typical family will
have equidistribution of signs of the functional equations. The minimalist conjecture on rank suggests that in the limit half should have rank $r$ and half rank $r+1$, giving an average rank of $r + 1/2$; however, in many families this is not observed. Instead, roughly $30\%$ have rank $r$ and $20\%$ rank $r+2$, while
about $48\%$ have rank $r+1$ and $2\%$ rank $r+3$. The question is whether or not the average rank stays on the order of $r + \foh + .40$ (or anything larger than $r + 1/2$, or if this is a result of small conductors and the limiting behavior not being seen. See \cite{Fe1,Fe2,Watkins} for numerical investigations and \cite{BhSh1,BhSh2,Br,H-B,FP,Michel95,Sil2,Yo2} for theoretical bounds of the average rank.

Consider families where the average second moment of $a_t(p)^2$ is $p^2 - m_\mathcal{E}p + O(1)$  with $m_\mathcal{E} > 0$, and let $t \in [N, 2N]$ for simplicity. We have already handled the contribution from $p^2$ to the 1-level density, and the $-m_\mathcal{E}p$ term contributes
\begin{eqnarray}\label{eq:psumforS2}
S_2 &\ \sim\ & \frac{-2}{N} \sum_p \frac{\log p}{\log
R}\widehat{\phi}\left(2\frac{\log p}{\log R}\right) \frac{1}{p^2}
\frac{N}{p}(-m_{\mathcal{E}}p) \nonumber\\ &=&
\frac{2m_{\mathcal{E}}}{\log R} \sum_p
\widehat{\phi}\left(2\frac{\log p}{\log R} \right) \frac{\log
p}{p^2}.
\end{eqnarray}
Thus there is a contribution of size $\frac{1}{\log R}$. A good
choice of test functions (see Appendix A of \cite{ILS}, or \cite{FrMil} for optimal test functions for all classical compact groups and larger support) is the
Fourier pair
\begin{eqnarray}
\phi(x) \ = \  \frac{\sin^2(2\pi \frac{\sigma}2 x)}{(2\pi x)^2}, \
\ \ \ \twocase{\widehat{\phi}(u) \ =\ }{\frac{\sigma - |u|}{4}}{if
$|u| \leq \sigma$}{0}{otherwise.}
\end{eqnarray}
Note $\phi(0) = \frac{\sigma^2}{4}$, $\widehat{\phi}(0) =
\frac{\sigma}4 = \frac{\phi(0)}{\sigma}$, and evaluating the prime
sum in \eqref{eq:psumforS2} gives
\begin{eqnarray}
S_2 & \ \sim \ & \left(\frac{.986}{\sigma} - \frac{2.966}{\sigma^2
\log R} \right)\frac{m_{\mathcal{E}}}{\log R}\ \phi(0).
\end{eqnarray} While we expect any $\sigma$ to hold, in all theoretical work to date $\sigma$ is greatly restricted. In \cite{Mil3} the consequences of this are analyzed in detail for various values of $\sigma$. If $\sigma = 1$ and
$m_{\mathcal{E}} = 1$, then the $1/\sigma$ term would contribute
$1$, the lower correction would contribute $.03$ for conductors of
size $10^{12}$, and thus the average rank is bounded by $1 + r +
\foh + .03$ $= r + \foh + 1.03$. This is significantly higher than
Fermigier's observed $r + \foh + .40$. If we were able to prove our $1$-level density for $\sigma = 2$,
then the $1/\sigma$ term would contribute $1/2$, and the
lower order correction would contribute $.02$ for conductors of size
$10^{12}$. Thus the average rank would be bounded by $1/2 + r +
1/2 + .02$ $= r + 1/2 + .52$. While the main error contribution is
from $1/\sigma$, there is still a noticeable effect from the
lower order terms in $A_{2}(p)$. Moreover, we are now in
the ballpark of Fermigier's bound; of course, we were already there
without the potential correction term!

\ \\

\section*{Acknowledgements}

This chapter is an extension of a talk given by the second named author at the Simons Symposium on Families of Automorphic Forms and the Trace Formula in Rio Grande, Puerto Rico, on January 28th, 2014. It is a pleasure to thank them and the organizers, as well as the anonymous referee who provided numerous suggestions which improved the paper (in particular the statement and proof of Theorem \ref{thm:dirichletallmRef}). Much of this work was conducted while the second named author was supported by NSF Grants DMS0600848, DMS0970067 and DMS1265673, and the others by NSF Grant DMS1347804 and Williams College.

\ \\

\appendix

\section*{Appendix A: Biases in Second Moments in Additional Families}\label{sec:biassecondmomentaddfamilies}

\noindent \texttt{By Megumi Asada, Eva Fourakis, Steven J. Miller and Kevin Yang} \ \\

This appendix describes work in progress on investigating biases in the second moments of other families. It is thus a companion to \S\ref{sec:secondmomentbiasEC}. Fuller details and proofs will be reported by the authors in \cite{AFMY};  our purpose below is to quickly describe results on analogues of the Bias Conjecture.

\subsection{Dirichlet Families}

Let $q$ be prime, and let $\mathcal{F}_q$ be the family of nontrivial Dirichlet characters of level $q$. In this family, the second moment is given by
\begin{align}
M_2(\mathcal{F}_q; X) \ = \ \sum_{p < X} \ \sum_{\chi \in \mathcal{F}_q} \chi^2(p).
\end{align}
Denote the amalgamation of families by $\mathcal{F}_{Y} = \cup_{Y/2 < q < Y} \mathcal{F}_q$, with the naturally defined second moment.

Computing $M_2(\mathcal{F}_q, X)$ is straightforward from the orthogonality relations, which as we've seen earlier yields a quantity related to the classical problem on the distribution of primes in residue classes. Approximating carefully $\pi(X)$ and $\pi(X, q, a)$ via the Prime Number Theorem, one can deduce the following.

\begin{thm}
The family $\mathcal{F}_q$ has positive bias, independent of $q$, in the second moments of the Fourier coefficients of the $L$-functions.
\end{thm}

\begin{rek} Note that the behavior of Dirichlet $L$-functions is very different than that from families of elliptic curves. \end{rek}

Now, suppose $q \neq \ell$ is a prime such that $q \equiv 1 (\ell)$. Let $\mathcal{F}_{q, \ell}$ be the family of non-trivial $\ell$-torsion Dirichlet characters of level $q$, which is nonempty by the stipulated congruence condition. In this family, the second moment is given by
\begin{align}
M_2(\mathcal{F}_{q,\ell}; X) \ = \ \sum_{p < X} \sum_{\chi \in \mathcal{F}_{q,\ell}} \chi^2(p).
\end{align}
Define $\mathcal{F}_Y := \cup_{Y/2 < q < Y} \mathcal{F}_{q,\ell_{q}}$ for any choice of suitable $\ell_q$ for each $q$.


\begin{thm}
The family $\mathcal{F}_{q,\ell}$ has zero bias independent of $q$ and $\ell$. Thus, $\mathcal{F}_{Y}$ exhibits zero bias in the second moments of the Fourier coefficients of the $L$-functions.
\end{thm}


\subsection{Families of Holomorphic Cusp Forms}

Let $S_{k,q}(\chi_0)$ denote the space of cuspidal newforms of level $q$, weight $k$ and trivial nebentypus, endowed with the structure of a Hilbert space via the Petersson inner product. Let $B_{k,q}(\chi_0)$ be any orthonormal basis of $S_{k,q}(\chi_0)$ and let $\mathcal{F}_{X} := \cup_{k < X: k \equiv 0 (2)} \mathcal{B}_{k, q = 1}(\chi_0)$. In this family, the second moment is given by the weighted Fourier coefficients\footnote{Following \cite{ILS} we can remove the weights, but their presence facilitates the application of the Petersson formula.}:
\begin{align}
M_2(\mathcal{F}_{X}; \delta) \ = \ \sum_{p < X^{\delta}} \ \sum_{k < X: k \equiv 0 (2)} \ \sum_{f \in B_{k, q}(\chi_0)} | \psi_{f}(p)|^2,
\end{align}
where $\psi_f(p) \ = \ \frac{\left(\Gamma(k-1)\right)^{\frac{1}{2}}}{(4 \pi p)^{\frac{k-1}{2}}} \lambda_f(p) \sqrt{\log p}$, with $\lambda_f(p)$ the Hecke eigenvalue of $f$ for the Hecke operator $T_p$. Let $\mathcal{F}_{X; \varepsilon} \ = \ \cup_{q < X^{\varepsilon}} \mathcal{F}_{X}$ be the amalgamation of families with the second moment
\begin{align}
M_2(\mathcal{F}_{X;\varepsilon}; \delta) \ = \ \sum_{p < X^\delta} \ \sum_{q < X^{\varepsilon}} \ \sum_{k < X: k \equiv 0 (2)} \ \sum_{f \in B_{k,q}(\chi_0)} \ | \psi_f(p)|^2.
\end{align}
The Petersson Formula provides an explicit method of computing $M_2(\mathcal{F}_X; \delta)$ via Kloosterman sums and Bessel functions. Averaging over the level and weight to obtain asymptotic approximations as in \cite{ILS}, we prove the following theorem in \cite{AFMY}.

\begin{thm}
The family $\mathcal{F}_{X}$ has negative bias, independent of the level $q$ of $\frac{1}{2}$, in the second moments of the Fourier coefficients of the $L$-functions. Thus, $\mathcal{F}_{X; \varepsilon}$ exhibits negative bias.
\end{thm}

Let us now let $H_{k,q}^*(\chi_0)$ denote a basis of newforms of Petersson norm 1 for prime level $q$ and even weight $k$. We consider another weighted second moment, given by
\begin{align}
M_2^{{\rm weighted}}(\mathcal{F}_{X}; \delta) \ = \ \sum_{p < X^{\delta}} \ \sum_{k < X: k \equiv 0 (2)} \ \sum_{f \in H_{k,q}^*(\chi_0)} \frac{\Gamma(k)}{(4 \pi)^k}|\lambda_f(p)|^2.
\end{align}
Similarly, let $\mathcal{F}_{X; \varepsilon} \ = \ \cup_{q < X^{\varepsilon}} \mathcal{F}_X$ be the amalgamation of these families with the weighted second moment
\begin{align}
M_2^{{\rm weighted}}(\mathcal{F}_{X;\varepsilon}; \delta) \ = \ \sum_{p < X^{\delta}} \ \sum_{q < X^{\varepsilon}} \ \sum_{k < X: k \equiv 0 (2)} \ \sum_{f \in H_{k,q}^*(\chi_0)} \ \frac{\Gamma(k)}{(4 \pi)^k} | \lambda_f(p)|^2.
\end{align}
We prove the following in \cite{AFMY}.

\begin{thm}
The family $\mathcal{F}_X$ has positive bias dependent on the level $q$. Moreover, the family $\mathcal{F}_{X, \varepsilon}$ exhibits positive bias as well.
\end{thm}

If we now consider the unweighted second moment given by
\begin{align}
M_2 (\mathcal{F}_X; \delta) \ = \ \sum_{p < X^{\delta}} \ \sum_{k < X: k \equiv 0 (2)} \ \sum_{f \in H_{k,q}^*(\chi_0)} \lambda_f^2(p),
\end{align}
we prove the following in \cite{AFMY} as well.

\begin{thm}
Assume $\delta < 1$ and $\varepsilon = 1$. The family $\mathcal{F}_X$ has positive bias dependent on $q$. Moreover, the family $\mathcal{F}_{X; \varepsilon}$ exhibits positive unweighted bias as well.
\end{thm}


\ \\
\end{document}